\documentclass[12pt,a4paper,twoside]{amsart}
\usepackage{amsmath,bbm,amssymb,amsxtra}
\usepackage{enumerate}
\allowdisplaybreaks
\usepackage{epsfig}
\usepackage{ae}
\theoremstyle{definition}

\newtheorem{lemma}{Lemma}[section]
\newtheorem{proposition}[lemma]{Proposition}

\newtheorem{theorem}[lemma]{Theorem}
\newtheorem{corollary}[lemma]{Corollary}
\newtheorem{remark}[lemma]{Remark}
\newtheorem{definition}[lemma]{Definition}

\newtheorem{question}[lemma]{Question}
\newcommand{\prop}[1]{\begin{proposition}\label{#1}
\sl }
\newcommand{\eprop}{\end{proposition}}
\newcommand{\thm}[1]{\begin{theorem}\label{#1}
\sl }
\newcommand{\ethm}{\end{theorem}}
\newcommand{\lem}[1]{\begin{lemma}\label{#1}
\sl }
\newcommand{\elem}{\end{lemma}}

\newcommand{\defin}[1]{\begin{definition}\label{#1}
\sl }
\newcommand{\edefin}{\end{definition}}

\newcommand{\beqno}{\begin{eqnarray*}}
\newcommand{\eeqno}{\end{eqnarray*}}
\newcommand{\beqla}[1] {\begin {eqnarray}\label{#1}}
\def\eeq {\end {eqnarray}}
\newcommand{\beq}{\begin {eqnarray}}

\newcommand{\real}{{\mathbb R}}

\newcommand{\integer}{{\mathbb Z}}
\newcommand{\nanu}{{\mathbb N}}
\newcommand{\complex}{{\mathbb C}}

\newcommand{\smooth}{C^\infty}
\newcommand{\distr}{{\mathcal D}'}

\newcommand{\sdistr}{{\mathcal S}'}
\newcommand{\dz}{\partial_z}
\newcommand{\dzb}{\partial_{\overline{z}}}

\newcommand{\torus}{{\mathbb T}}

\newcommand{\D}{\mathbb{D}}
\newcommand{\E}{\mathbb{E}}

\newcommand{\C}{\mathbb{C}}
\newcommand{\R}{\mathbb{R}}
\newcommand{\IH}{\mathbb{H}}
\newcommand{\IZ}{\mathbb{Z}}

\newcommand{\Chat}{\widehat{\mathbb{C}}}

\newcommand{\prob}{{\mathbb P}}
\newcommand{\expec}{{{\mathbb E}\,}}

\newcommand{\diam}{{{\rm diam}\,}}

       \def\CB{{\mathcal B}}       
\def\CD{{\mathcal D}}              
              \def\CI{{\mathcal I}}
\def\CJ{{\mathcal J}}              
\def\CM{{\mathcal M}}

\def\qq{ \begin{eqnarray} }
\def\qqq{ \end{eqnarray} }
\def\rr{ \begin{equation} }
\def\rrr{ \end{equation} }
\def\non{ \nonumber }
\newcommand{\NJ}{{{\bf J}}}

\def\qq{ \begin{eqnarray} }
\def\qqq{ \end{eqnarray} }
\def\non{ \nonumber }
\newcommand{\no}{\noindent}
\newcommand{\vs}{\vspace}

\newcommand{\hf}{{_1\over^2}}

\newcommand{\BbbN}{\mathbb{N}}

\newcommand{\BbbR}{\mathbb{R}}
\newcommand{\BbbC}{\mathbb{C}}

\newcommand{\BbbH}{\mathbb{H}}
\newcommand{\BbbD}{\mathbb{D}}%

\newcommand{\calF}{{\mathcal F}}
\newlength{\figheight}

\newcommand{\KHI}{\raisebox{0.07cm}{$\chi$}}

        \def\halmos{{\ \vbox{\hrule\hbox{\vrule
height1.3ex\hskip0.8ex\vrule}\hrule}}
            \par\medskip}
        
        \newcommand{\refeq}[1]{(\ref{#1})}
 
\begin{document}
\title[Random Conformal Weldings]
{Random Conformal Weldings}

\author[K. Astala]{Kari Astala$^1$}
\address{University of Helsinki, Department of Mathematics and Statistics,
         P.O. Box 68 , FIN-00014 University of Helsinki, Finland}
\email{kari.astala@helsinki.fi}


\author[P. Jones]{Peter Jones}
\address{Department of Mathematics,
Yale University, 10 Hillhouse Ave, New Haven, CT, 06510, U.S.A.}
\email{jones@math.yale.edu}

\author[A. Kupiainen]{Antti Kupiainen$^1$}
\address{University of Helsinki, Department of Mathematics and Statistics,
         P.O. Box 68 , FIN-00014 University of Helsinki, Finland}
\email{antti.kupiainen@helsinki.fi}

\author[E. Saksman]{Eero Saksman$^1$}
\address{University of Helsinki, Department of Mathematics and Statistics,
         P.O. Box 68 , FIN-00014 University of Helsinki, Finland}
\email{eero.saksman@helsinki.fi}
\footnotetext[1]{Supported by the Academy of Finland}


\begin{abstract}
We construct a conformally invariant random family of closed curves in the plane by
 welding of random homeomorphisms of the unit circle.
The   homeomorphism is constructed using the exponential of $\beta X$
where $X$ is the restriction of the two dimensional  free field on the circle  
and the parameter $\beta$ is in the "high temperature" regime  $\beta<\sqrt 2$.
The welding problem is solved by studying a  non-uniformly elliptic Beltrami equation with a random
complex dilatation. For the existence a method of  Lehto is used. This requires sharp probabilistic estimates to control conformal moduli of annuli and they are proven
by decomposing  the free field as a sum of independent fixed scale fields
and controlling the correlations of the complex dilation restricted to dyadic
cells of various scales. For uniqueness we invoke a result by Jones and Smirnov
on conformal  removability of H\"older curves. We conjecture that
our curves are locally related to SLE$(\kappa)$ for $\kappa<4$.
\end{abstract}

\maketitle


\section{Introduction}
\label{se:intro} 

There has been great interest in conformally
invariant random curves and fractals in the plane ever since it was
realized that such geometric objects appear naturally in statistical
mechanics models at the critical temperature \cite{Ca}. A major
breakthrough in the field occurred when O. Schramm \cite{Schra1}
introduced the Schramm-Loewner Evolution (SLE), a stochastic process
whose sample paths are conjectured (and in several cases proved) to
be the curves occurring in the physical models. We refer to
\cite{Schra2} and \cite{Smi} for a general overview and some recent
work on SLE. SLE curves come in two varieties: the radial one where
the curve joins a boundary point (say of the disc) to an interior
point and the chordal case where two boundary points are joined.

SLE describes a curve growing in time: the original curve of interest (say a cluster boundary in
a spin system) is obtained as time tends to infinity. In this paper we give a different construction
of random curves which is stationary i.e. the probability measure on curves is directly defined without
introducing an auxiliary time. We carry out this construction for closed curves, a case that is
not naturally covered by SLE.

Our construction is based on the idea of conformal welding. Consider a
Jordan curve $\gamma$ bounding a
simply connected region $\Omega$ in the plane. By the Riemann
mapping theorem there are conformal maps $f_\pm$ mapping the
 unit disc $\BbbD$  and its complement to $\Omega$
and its complement. The map $f_+^{-1}\circ f_-$ extends continuously to the
boundary $\torus = \partial \D$ of the disc and defines  a homeomorphism of the
circle. Conformal Welding is the inverse operation where, given a
suitable homeomorphism of the circle, one constructs a Jordan curve
on the plane (see Section 2). In fact, in our case the curve is determined up to a
 M\"obius transformation of the plane.
Thus random curves (modulo M\"obius transformations) can be obtained
from random homeomorphisms via welding.

In this paper we introduce a  random scale invariant set of
homeomorphisms $h_\omega: \torus\to \torus$ and construct the welding
curves.  The model considered  here has been proposed by the second
author. The construction depends
 on a real parameter $\beta$ ("inverse temperature") and the maps are a.s. in $\omega$
 H\"older continuous for $\beta<\beta_c$. For this range of $\beta$
  the welding map will be a.s.  well-defined. For$\beta>\beta_c$ we expect the
  map  $h_\omega$ not to be continuous and no welding to exist.
  We conjecture that the
  resulting curves should locally "look like" SLE$(\kappa)$ for $\kappa<4$
  but we don't have good arguments for this.
  The case $\beta=\beta_c$,
  presumably corresponding to  SLE$(4)$, is not covered by our analysis.

 Since we are interested in random curves that are self similar it is natural to
 take $h$ with such properties. Our choice for $h$ is constructed by starting with the
 Gaussian random field $X$  on the circle (see Section 3 for precise definitions)
 with covariance
\begin{equation}
\E\, X(z)X(z')=-\log |z-z'|
\label{co1}
\end{equation}
where $z,z'\in \BbbC$ with modulus one.  $X$ is just the
restriction of the 2d massless free field (Gaussian Free Field) on
the circle. The exponential of $\beta X$ gives rise to a random measure $\tau$ on the unit circle $\torus$,
formally given by
\begin{equation}
"d\tau= e^{\beta X(z)}dz".
\label{f1}
\end{equation}
The proper definition 
involves a limiting process $\tau (dz)=\lim_{\varepsilon\to 0}
e^{\beta X_\varepsilon (z)}/\expec e^{\beta X_\varepsilon (z)} dz$, where $X_\varepsilon$ stands for a suitable regularization
of $X,$ see Section \ref{subse:exponential} below.

Identifying the circle as $\torus =\real /\integer =[0,1)$ our 
 random homeomorphism  $h:[0,1)\to [0,1)$ is  defined as
  \begin{equation}
h(\theta)=\tau ([0,\theta ))/\tau ([0,1) )\qquad \mbox{for}\;\; \theta\in [0,1).
\label{h1}
\end{equation}

\medskip

\noindent  The main result of this paper can then be summarized as follows: 
\smallskip

\noindent {\it For
$\beta^2< 2$ and almost surely in $\omega$, the formula (\ref{h1})
defines a H\"older continuous 
circle homeomorphism, such that  the welding
problem has a  solution $\gamma$, where  $\gamma$ is a  Jordan curve  bounding a domain $\Omega = f_+(\BbbD)$ with
a  H\"older continuous Riemann mapping $f_+$. For a given $\omega$, the solution is unique up to a M\"obius map of the plane.} 
\smallskip

\noindent We refer to Section
 \ref{se:conclusion} (Theorem \ref{main result}) for the exact statement
 of the main result.
 \medskip
 
 Apart from connection to SLE the weldings constructed in this paper should be of interest to complex analysts as they form a natural family that degenerates as $\beta\uparrow\sqrt{2}.$ It would be of great interest to understand the critical case $\beta=\sqrt{2}$ as well as the low
 temperature "spin glass phase" $\beta>\sqrt{2}.$ It would also be of interest to understand the connection of our weldings to those arising from stochastic flows \cite{AiMaTha}. In \cite{AiMaTha}  H\"older continuous homeomorphisms are considered, but the boundary behaviour
 of the welding and hence its existence and uniqueness are left open.

In writing the paper we have tried be generous in providing details on both  the function theoretic and the stochastics notions and tools needed, in order to serve readers with varied backgrounds.
The structure of the paper is as follows. Section
\ref{se:welding} contains background
material on conformal welding and the geometric analysis tools we
need later on. To be more specific, Section \ref{se:welding}
recalls the notion of conformal welding and explains  how the
welding problem is reduced to the study of the Beltrami equation.
Also we recall  a useful method due to Lehto \cite{Le} to prove the
existence of  a solution for a class of non-uniformly elliptic
Beltrami equations, and a theorem by Jones and Smirnov \cite{JoSmi}
that will be used to verify  the uniqueness of our welding. Finally we recall
 the Beurling-Ahlfors
extension of homeomorphisms of the circle to the inside the disc.
For our purposes we need to estimate carefully the dependence of the
dilatation of the extension in a Whitney cube by just using small
amount of information of the homeomorphism on a 'shadow' of the
cube.

 In Section \ref{se:circlehomeo} we introduce the one-dimensional
trace of the Gaussian free field and recall some known properties of
its exponential  that we will use to define and study the random
circle homeomorphism.
Section \ref{se:probability} is the technical core of the paper
 as it contains the main probabilistic estimates we need
to control the random dilatation of the extension map. Finally,  in Section
\ref{se:conclusion}  things are put together and the
a.s. existence and uniqueness of the welding map is proven.

Let us finish by a remark on notation. We denote by $c$ and $C$
generic constants which may vary between estimates. When 
the constants depend on parameters such as $\beta$ we denote
this by $C(\beta)$.

\vskip 3mm

\no {\bf Acknowledgements}. We thank M. Bauer, D. Bernard,  Steffen Rohde and Stanislav Smirnov
for useful discussions. This work is partially funded by the Academy of Finland,
European Research Council and National Science Foundation.

\section{Conformal Welding}
\label{se:welding}

In the present section we recall for the general readers benefit
basic notions and results from geometric analysis that are needed in
our work. In particular, we recall the notion of conformal welding,
Lehto's method for solving the Beltrami-equations, the uniqueness
result for weldings due to Jones and Smirnov, and the last
subsection contains estimates for the Beurling-Ahlfors extension
tailored for our needs.

\subsection{Welding and Beltrami equation}

One of the main methods for constructing conformally invariant
families of (Jordan) curves comes from the theory of {\it conformal
welding}. Put briefly, in this method we glue the unit disk $\D =\{z
\in \C : |z| < 1\}$ and the exterior disk $\D_\infty = \{z \in \Chat
: |z| > 1\}$ along a homeomorphism $\phi: \torus\to \torus$, by the
identification
$$x \sim y, \quad\;  \mbox{ when } x \in \torus = \partial \D, \;\mbox{ } y = \phi(x) \in \torus = \partial \D_\infty.
$$
The problem of welding is to give a natural complex structure to
this topological sphere. Uniformizing  the
complex structure then gives us the curve, the image of the unit
circle.

More concretely,  given a Jordan curve  $\Gamma
\subset \Chat$, let
$$ f_+ : \D \to \Omega_+ \;\; \mbox{ and } \;\; f_- : \D_\infty \to \Omega_-
$$
be  a choice of Riemann mappings onto the components of the
complement $\Chat \setminus \Gamma = \Omega_+ \cup \Omega_-$. By
Caratheodory's theorem $f_-$ and $f_+$ both extend continuously to
$\partial \D = \partial \D_\infty$, and thus \beqla{homeo11}
 \phi = \left( f_+ \right)^{-1} \circ f_-
\eeq is a homeomorphism of $\torus$. In the welding problem we are asked
to invert this process; given a homeomorphism $\phi: \torus \to \torus$ we
are to find a Jordan curve $\Gamma$ and conformal mappings $f_{\pm}$
onto the complementary domains $\Omega_{\pm}$ so that
(\ref{homeo11}) holds. 



It is clear that the welding problem, when solvable, has natural
conformal invariance attached to it; any image of the curve $\Gamma$
under a M\"obius transformation of $\Chat$ is equally a welding
curve. Similarly, if $\phi: \torus \to \torus$ admits a welding, then so do
all its compositions with M\"obius transformations of the disk.
Note, however, that not all circle homeomorphisms admit a welding,
for examples see \cite{Oi} and \cite{Va}.

\bigskip

The most powerful tool in solving the welding problem is given by
the Beltrami differential equation, defined in a domain $\Omega$ by
\beqla{homeo12} \frac{\partial f}{\partial \overline z} = \mu(z)
\frac{\partial f}{\partial  z}, \quad \quad \mbox{ for almost every
} z \in \Omega, \eeq where we look  for homeomorphic solutions $f
\in W^{1,1}_{loc}(\Omega)$. Here (\ref{homeo12}) is an elliptic
system whenever $|\mu(z)| < 1$ almost everywhere, and uniformly
elliptic if there is a constant  $0 \leq k< 1$  such that $\| \mu
\|_\infty \leq k $.

In the uniformly elliptic case,  homeomorphic solutions to
(\ref{homeo12}) exist for every coefficient with  $\| \mu \|_\infty
 < 1$, and  they are unique up to post-composing with a
conformal mapping \cite[p.179]{AsIwMa}. In fact, it is this
uniqueness property that gives us a way to produce the welding.
To see this
suppose first that \beqla{homeo13}
 \phi = f|_{\torus},
\eeq where $f\in W^{1,2}_{loc}(\D;\D)\cap C(\overline{\D})$ is a
homeomorphic solution to (\ref{homeo12}) in the disc $\D$. Find then
a homeomorphic solution to the auxiliary equation \beqla{homeo14}
\frac{\partial F}{\partial \overline z} =
\KHI_{\D}(z) \, \mu(z) \; \frac{\partial
F}{\partial  z}, \quad \quad \mbox{ for } a.e. \; z \in \C. \eeq Now
$\Gamma = F(\torus)$ is a Jordan curve. Moreover, as
$\partial_{\overline z} F = 0$ for $|z| > 1$, we can set $ f_- :=
F|_{\D_\infty} \mbox{ and }   \Omega_- := F(\D_\infty)$ to define a
conformal mapping
$$ f_-: \D_\infty \to \Omega_-$$
On the other hand, since both $f$ and $F$ solve the equation
(\ref{homeo12}) in the unit disk $\D$, by uniqueness of the
solutions we have \beqla{homeo31}
 F(z) = f_+ \circ f(z), \quad z \in \D,
\eeq for some conformal mapping $f_+ : \D = f(\D) \to \Omega_+ :=
F(\D)$. Finally, on the unit circle,
\beqla{homeo32}
 \hskip40pt \phi(z)  = f|_{\torus}(z) =  \left( f_+ \right)^{-1} \circ f_-(z), \quad \quad z \in \torus.
\eeq
Thus we have found a solution to the welding problem, under the
assumption (\ref{homeo13}). That the welding curve $\Gamma$ is
unique up to a  M\"obius transformation of $\C$ follows from
\cite[Theorem 5.10.1]{AsIwMa}, see also Corollary \ref{homeo42}
below.

To complete this circle of ideas we need to identify the
homeomorphisms $\phi: \torus \to \torus$ that admit the representation
(\ref{homeo12}),  (\ref{homeo13}) with uniformly elliptic $\mu$ in (\ref{homeo12}). It turns out \cite[Lemma 3.11.3
and Theorem 5.8.1]{AsIwMa} that such $\phi$ 's are precisely the
quasisymmetric mappings of $\torus$, mappings that satisfy
\beqla{homeo17}\label{qscond}
 \hskip40pt K(\phi) := \sup_{s, t \in \R} \frac{|\phi(e^{2\pi
i(s+t)}) - \phi(e^{2\pi i s}) |}{|\phi(e^{2\pi i(s-t)}) -
\phi(e^{2\pi i s})|}  < \infty. \eeq

\bigskip

\subsection{Existence in the degenerate case: the Lehto condition}

The previous subsection  describes an obvious model for constructing
random Jordan curves, by first finding  random homeomorphisms of the
circle and then solving for each of them  the associated welding
problem. In the present work, however, we are faced with the
obstruction that circle homeomorphisms with derivative the
exponentiated Gaussian free field almost surely do not satisfy the
quasisymmetry assumption (\ref{homeo17}). Thus we are forced outside
the uniformly elliptic PDE's and need to study (\ref{homeo12}) with
degenerate coefficients with only $|\mu(z)| < 1$ almost everywhere.
We are even outside the much studied class of maps of exponentially
integrable distortion, see \cite[ 20.4.]{AsIwMa} In such generality,
however, the homeomorphic solutions to (\ref{homeo12}) may fail to
exist, or the crucial uniqueness properties of (\ref{homeo12}) may
similarly fail.
\medskip

In his important work \cite{Le} Lehto gave a very  general condition
in the degenerate setting, for  the existence of homeomorphic
solutions to (\ref{homeo12}). To recall his result, assume we are
given  the complex dilatation  $\mu=\mu (z)$, and  write then
$$K(z)  = \frac{1+|\mu(z) |}{1-|\mu(z) |}, \quad z \in \Omega,
$$
for  the associated distortion function. Note that $K(z)$ is bounded
precisely when the equation  (\ref{homeo12}) is uniformly elliptic,
i.e. $\| \mu \|_{\infty}  < 1$. Thus the question is how
strongly can $K(z)$  grow for the basic properties of
(\ref{homeo12}) still to remain true. In order to state Lehto's
condition we fix some notation. For given $z\in\complex $ and radii
$0\leq r< R<\infty$ let us denote the corresponding annulus by
$$
A(z,r,R):=\{  w\in\complex : r<|w-z|<R\} .
$$
In the Lehto approach one needs to control the conformal moduli of
image annuli in a suitable way. This is done by introducing for any
annulus  $A(w,r,R)$ and for the given distortion function $K$ the
following quantity, which we call {\it the Lehto integral}:
\beqla{eq:4.7} L(z,r,R):=L_K(z,r,R):=\int_r^R  \, \frac{1 }{
\int_0^{2\pi}
 {K}\left(z +\rho e^{i\theta}\right) d\theta} \; \frac{d\rho}{\rho}
\eeq

For the following formulation of Lehto's theorem see \cite[p.
584]{AsIwMa}.
\bigskip

\begin{theorem} \label{viimet44} Suppose  $\mu$ is  measurable and  compactly supported with
 $|\mu(z)| < 1$ for almost every $z \in \C$.
Assume that the distortion function  $K(z) = (1+ |\mu(z)|)/(1-
|\mu(z)|)$ is locally integrable,
\begin{equation} \label{viimet127}
\,{K} \in L^1_{loc}(\C),
 \end{equation}
  and  that for some $R_0>0$ the Lehto integral satisfies
 \begin{equation} \label{viimet27}  L_K(z,0,R_0) = 
 \infty,  \quad \mbox{for all}\quad z \in \C.
 \end{equation}
Then the Beltrami equation
\begin{equation} \label{viimet34}
  \frac{\partial f}{\partial \overline z}  (z) = \mu (z)  \frac{\partial f}{\partial z}   (z) \quad \mbox{ for almost every $z\in \C,$}
 \end{equation}
admits a homeomorphic $W^{1,1}_{loc}$-solution $f:\C \to \C$.
\end{theorem}
\medskip

\begin{corollary} Suppose $\phi:\torus \to \torus$  extends to a homeomorphism
$f:\C \to \C$ satisfying   (\ref{viimet127}) - (\ref{viimet34})
together with the condition
 \begin{equation} \label{viimet128}
K(z) \in L^{\infty}_{loc}(\D).
 \end{equation}
 Then  $\phi$ admits a welding: there
are a Jordan curve $\Gamma \subset \Chat$ and conformal mappings
$f_\pm$ onto the complementary domains of $\Gamma$ such that
$$ \phi(z)  =  \left( f_+ \right)^{-1} \circ f_-(z), \quad  z \in \torus.
$$
\end{corollary}

\begin{proof} Given the extension $f: \C \to \C$ let us again look at the auxiliary equation
\beqla{homeo114} \frac{\partial F}{\partial \overline z} =
\KHI_{\D}(z) \, \mu(z) \;  \frac{\partial
F}{\partial  z}, \quad \quad \mbox{ for } a.e. \; z \in \C. \eeq
Since Lehto's condition holds as well for the new distortion function
$$K(z) = \frac{1+| \KHI_{\D}(z) \mu(z) |}{1-|\KHI_{\D}(z) \mu(z) |},
$$
we see from Theorem \ref{viimet44} that the auxiliary equation
(\ref{homeo114}) admits a homeomorphic solution $F:\C \to \C$.
Arguing as in (\ref{homeo13}) - (\ref{homeo32}) it will be then
sufficient to show that
$$ F(z) = f_+ \circ f(z), \quad z \in \D,
$$
where $f_+$ is conformal in $\D$. But this is a local question;
every point $z \in \D$ has a neighborhood  where $K(z)$ is uniformly
bounded, by (\ref{viimet128}). In such a neighborhood  the usual
uniqueness results to solutions  of (\ref{homeo12}) apply, see
\cite[p.179]{AsIwMa}. Thus $f_+$ is holomorphic, and as a
homeomorphism it is conformal. This proves the claim.  \halmos
\end{proof}

\medskip

Consequently, in the study of random circle  homeomorphisms $\phi =
\phi_\omega$ a key step for the conformal  welding of $\phi_\omega$
will be to show that  almost surely each such mapping admits a
homeomorphic  extension to $\C$, where the distortion function
satisfies a condition such as (\ref{viimet27}).  In our setting
where derivative of  $\phi$ is given by the exponentiated trace of a
Gaussian free field,  the extension procedure is described in Section \ref{subse:extension} and the  appropriate estimates it requires are proven in Section
\ref{se:probability}.

 Actually,  in Section \ref{se:conclusion} when
proving  our main theorem we need to present a variant of Lehto's
argument where it will be enough to estimate the Lehto integral only
at a suitable countable set of points $z\in \torus$.  We also utilize
there the fact that the extension of our random circle homeomorphism
$\phi$ satisfies \refeq{viimet128}.
 In verifying the H\"older
continuity of the ensuing map we shall apply a useful estimate
(Lemma \ref{le:4.1} below) that estimates the geometric distortion
of an annulus under a quasiconformal map. 

Given a bounded
(topological) annulus $A\subset\complex$ , with $E$ the bounded
component of $\complex\setminus A,$
 we denote by $D_O(A):=\diam (A)$ the outer diameter,
  and by $D_I(A):=\diam (E)$ the inner diameter of A.

\lem{le:4.1} Let $f $ be a quasiconformal mapping on the annulus $A(w,r,R)$, 
with the distortion function $K_f.$ It then
holds that
$$
\frac{D_O(f(A(w,r,R)))}{D_I(f(A(w,r,R)))}\geq
\frac{1}{16}\exp\left(2\pi^2 L_{K_f}(w,r,R)\right).
$$
\elem
\begin{proof} Recall first that for a rigid annulus $A = A(w,r,R)$ the modulus 
$$mod(A) = 2 \pi \log \frac{R}{r}
$$
while for any topological annulus $A$,  we define its  conformal modulus by $mod(A) = mod(g(A))$ where $g$ is a conformal map of $A$ onto a rigid annulus.
Then we have   \cite[Cor. 20.9.2]{AsIwMa}  the following  basic estimate for the modulus of
the image  annulus in terms of the Lehto integrals: \beqla{eq:4.16} mod(f(A(w,r,R)))\geq 2\pi
L_{K_f}(w,r,R). \eeq  On the other hand, by combining \cite[7.38
and 7.39]{Vu} and \cite[5.68(16)]{AnVaVu} we obtain for any bounded
topological annulus $A\subset\complex$
$$
\frac{1}{16}\exp (\pi \, mod(A))\leq \frac{D_O(A)}{D_I(A)}.
$$
Put together, the desired estimate follows. \halmos
\end{proof}

\bigskip

\subsection{Uniqueness of the welding}
\label{subse:uniqueness}

An important issue of the welding is its uniqueness, that the curve
$\Gamma$ is unique up to composing with a M\"obius transformation of
$\Chat$. As the above argument indicates, this is essentially
equivalent to the uniqueness of solutions to the appropriate
Beltrami equations, up to a M\"obius transformation. However, in
general the assumptions of Theorem \ref{viimet44} alone are much too
weak to imply  this.

It fact, in our case the uniqueness of solutions to the Beltrami
equation (\ref{homeo114})
 is equivalent to the conformal removability of the curve $F(\torus)$. Recall  that a compact set $B \subset \Chat$ is {\it conformally removable} if every homeomorphism of $\Chat$ which is conformal off $B$ is  conformal in the whole sphere, hence a M\"obius transformation.

It follows easily that e.g. images of circles under quasiconformal
mappings, i.e. homeomorphisms satisfying (\ref{homeo2}) with $\| \mu
\|_\infty  < 1$,  are conformally removable,  while Jordan
curves of positive area are never conformally removable.

For general curves the removability is a deep problem; no
characterizations of conformally removable Jordan curves  is known
to this date. What saves us in the present work  is that we have
available the remarkable result of Jones and Smirnov in \cite{JoSmi}.
We will not need their result in its full generality, as the
following special case will sufficient for our purposes.

\begin{theorem} \label{homeo41} ({\rm Jones, Smirnov \cite{JoSmi}}) Let $\Omega \subset \Chat$ be a simply connected domain such that the Riemann mapping $\psi:\D \to \Omega$ is $\alpha$-H\"older continuous for some $\alpha > 0$.

Then the  boundary $\partial \Omega$ is conformally removable.
\end{theorem}
\medskip

\noindent Adapting this result to our  setting we obtain
\medskip

\begin{corollary} \label{homeo42}  Suppose $\phi: \torus \to \torus$ is a homeomorphism that admits a welding
$$ \phi(z)  =  \left( f_+ \right)^{-1} \circ f_-(z), \quad  z \in \torus,
$$
where $f_\pm$ are conformal mappings of $\D$ and $\D_\infty$,
respectively, onto complementary Jordan domains $\Omega_\pm$.

Assume that $f_-$ (or $f_+$) is $\alpha$-H\"older continuous on the
boundary $\partial \D_\infty = \torus$. Then the welding is unique: any
other welding pair $(g_+, g_-)$ of $\phi$ is of the form
$$ g_\pm = \Phi \circ f_\pm, \quad \quad \Phi:\Chat \to \Chat \mbox{ M\"obius.}
$$
\end{corollary}

\begin{proof} Suppose we have Riemann mappings $g_\pm$ onto complementary Jordan domains such that
 $$  \left( g_+ \right)^{-1} \circ g_-(z) = \phi(z) = \left( f_+ \right)^{-1} \circ f_-(z), \quad  z \in \torus.
 $$
 Then the formula
 $$\Psi(z) = \left\{\begin{array}{ll}
g_+ \circ  \left( f_+ \right)^{-1}(z) &\quad \mbox{if $z\in f_+(\D)$} \\
g_- \circ  \left( f_- \right)^{-1}(z)  &\quad \mbox{if $z\in
f_-(\D_\infty)$}
\end{array} \right.
 $$
 defines a homeomorphism of $\Chat$ that is conformal outside $\Gamma = f_\pm(\torus)$. From the Jones-Smirnov theorem we see that $\Psi$ extends conformally to the entire sphere; thus it is a M\"obius transformation.
\halmos \end{proof}
\medskip

As we shall see in Theorem \ref{th:4.1}, for circle  homeomorphisms
$\phi$ with derivative the exponentiated Gaussian free field, the
solutions $F\;$ to the auxiliary equation (\ref{homeo114}) will be
H\"older continuous  almost surely. Then $f_-= F|_{\D_\infty}$ is a
Riemann mapping onto a complementary component of the welding curve
of $\phi = \phi_\omega$. It  follows  that almost surely the  $\phi
= \phi_\omega$ admits a welding curve $\Gamma = \Gamma_\omega$
which is unique, up to composing with a M\"obius transformation.

\bigskip

\subsection{Extension of the homeomorphism}
\label{subse:extension}

In this section we  discuss  in detail
suitable methods of extending homeomorphisms $\phi: \torus \to \torus$ to
the unit disk; by reflecting across $\torus$ the map then extends to
$\C$.   Extensions of homeomorphisms $h:\R \to \R$ of the real line
are convenient to describe, and it is not difficult to find
constructions that sufficiently well  respect the conformally
invariant features of $h$. Given a homeomorphism $\phi: \torus \to \torus$
on the circle, we hence represent it in the form \beqla{homeo1}
 \phi(e^{2\pi i x}) = e^{2\pi i h(x)}
\eeq where  $h:\R \to \R$ is a  homeomorphism of the line with
$h(x+1) = h(x)+1$. We may assume that $\phi(1) = 1$, with $ h(0) =
0$. 



 We will now  extend the 1-periodic mapping $h$  to the upper  (or lower) half plane
 so that it becomes the identity map at large height. Then  a conjugation to a mapping of the disk is easily done. For the extension we use
 the classical Beurling-Ahlfors extension \cite{BA}
 modified suitably
far away from the real axis.

Thus, given a   homeomorphism   $h:\R \to \R$ such that \beqla{homeo2}
h(x+1) = h(x)+1, \quad x \in \R, \quad \mbox{with } h(0) = 0, \eeq
we   define our extension $F$ as follows. For $0 < y  < 1$ let
\beqla{BA}
F(x+iy) = \frac{1}{2} \int_0^1\bigl(  h(x+ty) + h(x-ty) \bigr) dt \hskip30pt \\
 +\;  i \int_0^1\bigl(  h(x+ty) - h(x-ty) \bigr) dt. \nonumber
\eeq Then $F=h$ on the real axis, and $F$ is a continuously
differentiable homeomorphism.

  Moreover, by \refeq{homeo2} it follows that for $y = 1$,
$$F(x+i) = x + i + c_0,
$$
where $c_0 = \int_0^1h(t)dt -1/2 \in [-1/2,1/2]$. Thus for $1 \leq y
\leq 2$ we set
  \beqla{ext17}
  F(z) = z + (2-y) c_0, 
  \eeq
and finally  have an extension of $h$ with the extra properties
  \beqla{homeo3}
F(z) \equiv z \quad \mbox{when } y = \Im m (z)  \geq 2, \eeq
    \beqla{homeo5}
F(z+k) = F(z)+k, \quad k  \in \IZ. \eeq

The original circle mapping admits a natural extension to the disk,
   \beqla{homeo51}
    \Psi(z) = \exp \bigl(2\pi i \,  F (\log z \,/ \, 2\pi i)  \bigr), \quad z \in \D.
 \eeq
  From (\ref{homeo1}), (\ref{homeo5}) we see that  this is a well defined homeomorphism of the disk with $\Psi |_{\torus} = \phi$ and  
   $\Psi(z) \equiv z$ for $|z| \leq e^{-4 		\pi}$.
   The distortion properties are not altered under this locally conformal change of variables,
     \beqla{homeo151}
    K(z, \Psi) = K(w, F), \quad \; z = e^{2\pi i\, w}, \;w \in \IH,
 \eeq
so we will reduce all distortion estimates for $\Psi$ to  the corresponding ones
 for $F$. Since $F$ is conformal for $y>2$ it suffices to  restrict the
 analysis to the strip 
  \beqla{Sdef}
  S=\R\times [0,2].
  \eeq
  
     To estimate $ K(w, F)$ we introduce some notation. Let
 $$ {\mathcal D}_n =  \{ \, [k 2^{-n}, (k+1) 2^{-n}] : k \in \IZ \}$$
 be the set of all dyadic intervals of length $2^{-n}$ and write
 $${\mathcal D}  = \{ {\mathcal D}_n : n \geq 0 \}.
 $$
 Consider  the measure
 $$\tau([a,b]) = h(b)-h(a).
 $$
 For a pair of intervals  $\NJ=\{J_1,J_2\}$ 
 let us
 introduce the following quantity
\begin{equation}
\delta_\tau(\NJ)=\tau(J_1)/\tau(J_2)+\tau(J_2)/\tau(J_1) \label{delta}.
\end{equation}
If  $J_i$ are the two halves  of an interval $I$,  
then
$\delta_\tau(\NJ)$ measures the local doubling properties of the measure $\tau$. 
In  such a case
we define $\delta_\tau(I) = \delta_\tau(\NJ)$. In particular, (\ref{qscond}) holds for the circle homeomorphism $ \phi(e^{2\pi i x}) = e^{2\pi i h(x)}$ if and only if the quantities $\delta_\tau(I)$
are uniformly bounded, for all (not necessarily dyadic) intervals $I$.
\medskip

The local
distortion of the extension $F$  will  be controlled by sums of the expressions
$\delta_\tau(\NJ)$ in the appropriate scale. For this, let us pave the strip $S$ by  
Whitney cubes $\{C_I\}_{I\in\mathcal D}$
defined by
$$C_I = \{ (x,y) : \; x\in I, \; 2^{-n-1} \leq y \leq   2^{-n}\} $$
for 
 $I \in{\mathcal D}_n$, $n>0$ and $C_I=I\times [\hf,2]$ for  $I \in{\mathcal D}_0$.
 Given an $I\in\mathcal D_n$ let $j(I)$ be the union of $I$ and its neighbours in ${\mathcal D}_n$ and 
  \beqla{calJ}
 \CJ(I):=\{\NJ=(J_1,J_2): J_i\in{\mathcal D}_{n+5}, J_i\subset  j(I)\}.
  \eeq
We define then
 \begin{equation}
\hskip40pt {K}_\tau(I) := \sum_{\NJ\in\CJ(I)}\delta_\tau(\NJ). \label{Kbound}
\end{equation}

With these notions we have the basic geometric estimate for the
distortion function, in terms of the boundary homeomorphism:
\medskip

 \begin{theorem} \label{alaReed}
 Let $F:\IH \to \IH$ be the extension of a $1$-periodic homeomorphism $h:\R \to \R$. Then  for each $I \in {\mathcal D}$
 \begin{equation}\label{kolmiot}
 \sup_{z \in C_I} K(z,F) \leq C_0 \, {K}_\tau(I), \quad \quad \; 
\end{equation}
with a universal constant $C_0$.
\end{theorem}
\begin{proof}The distortion properties of the Beurling-Ahlfors extension are well studied in the existing litterature, 
but none of
these works gives directly Theorem 2.6 as the main point for us is the linear dependence on the local distortion ${K}_\tau(I)$. The most elementary extension operator is due to Jerison and Kenig \cite{JeKe}, see also \cite[Section 5.8]{AsIwMa}, but for this extension the linear dependence fails.

For the reader's convenience we sketch a proof for the theorem. We will
modify  the approach of Reed \cite{Re}, and start with the simple
Lemma.

 \begin{lemma} \label{apureed}
 For each dyadic interval $I = [k 2^{-n}, (k+1) 2^{-n}]$, with left half $I_1 = [k2^{-n},(k+1/2)2^{-n}]$ and right half $I_2 = I \setminus I_1$, we have
 $$\frac{1}{1+\delta_\tau(I)}  \;|\tau(I)|  \leq \, |\tau(I_1)|, \, |\tau(I_2)| \,  \leq \frac{\delta_\tau(I)}{1+\delta_\tau(I)} \;|\tau(I)|$$ with
$$ \frac{1}{|I|} \int_I h(t) - h(k 2^{-n}) dt \leq \frac{3\delta_\tau(I)}{1+3\delta_\tau(I)} \;|\tau(I)|$$ and
$$\frac{1}{|I|} \int_I h((k+1) 2^{-n}) -h(t) dt \leq \frac{3\delta_\tau(I)}{1+3\delta_\tau(I)}\; |\tau(I)|$$
 \end{lemma}
 \begin{proof} The definition of $\delta_\tau(I)$ gives
directly  the first estimate. As $h(t) \leq h((k+1/2)2^{-n})$ on
the left half and $h(t) \leq h((k+1) 2^{-n})$ on the right half of
$I$,
 $$ \frac{1}{|I|} \int_I h(t) - h(k 2^{-n}) dt \leq \left(\frac{1}{2} \frac{\delta_\tau(I)}{1+\delta_\tau(I)} + \frac{1}{2}\right) |\tau(I)| \leq \frac{3\delta_\tau(I)}{1+3\delta_\tau(I)}\; |\tau(I)|.
 $$
 The last estimate follows similarly.
 \halmos \end{proof}

To continue with the proof of Theorem \ref{alaReed}, the pointwise
distortion of the extension $F$ is easy to calculate explicitly, and
we obtain \cite{BA, Re} the following estimate, sharp up to a
multiplicative constant,
  \begin{equation}\label{distortioarvio}
 K(x+ iy,F) \leq \left( \frac{\alpha(x,y)}{\beta(x,y)} + \frac{\beta(x,y)}{\alpha(x,y)} \right) \left[ \frac{\tilde\alpha(x,y)}{\alpha(x,y)} + \frac{\tilde\beta(x,y)}{\beta(x,y)} \right]^{-1},
\end{equation}
  where
  $$\alpha(x,y) = h(x+y) - h(x), \quad \beta(x,y) = h(x) - h(x-y)
  $$
  and
  $$\tilde\alpha(x,y) = h(x+y) - \frac{1}{y}\int_x^{x+y}h(t) dt, \quad \tilde\beta(x,y) =  \frac{1}{y}\int_{x-y}^xh(t) dt\, -\, h(x-y).
  $$
Now the argument of Reed \cite[pp. 461-464]{Re}, combined with
Lemma \ref{apureed} and its estimates, precisely shows that
$K(x+iy,F) \leq 24\max \delta_\tau(\tilde I)$, where $\tilde I$ runs
over the intervals with endpoints contained in the set
  \begin{equation}\label{distortioarvio7}  \{x, x\pm y/4, x \pm y/2, x \pm y \}.
\end{equation}
\smallskip

Thus, for example, if we fix $k\in \IZ$ and $n\in \nanu$, we get for
the corner point $z = k 2^{-n} + i 2^{-n}$ of the Whitney cube
$C_I$ the estimate
  \begin{equation}\label{distortioarvio2}
 K( k 2^{-n} + i 2^{-n},F) \leq 24 \sum_{\NJ}\delta_\tau(\NJ), \quad \NJ=(J_1,J_2),\;\; J_i\in{\mathcal D}_{n+3},
\end{equation}
 where $J_i\subset j(I)$ as above.
  For a general point $z=x+iy \in C_I$, we have to take a few more generations of dyadic intervals. Here   $[x,x+y/4]$  has length at least $2^{-n-3}$. On the other hand, for any (non-dyadic) interval $\tilde I$ with $2^{-m} \leq |\tilde I | < 2^{-m+1}$, one observes that it contains a dyadic interval of length $2^{-m-1}$ and is contained
inside  a union of at most three dyadic intervals of length  $2^{-m}.$ By this manner one estimates 
$$\delta_\tau(\tilde I) \leq \sum_{\tilde {\bf J}}\delta_\tau( {\bf J}), \; \mbox{where }  { \bf J} = (J_1,J_2),\;\; J_i\in{\mathcal D}_{m+2} \;\mbox{ and } \; J_i \cap \tilde I \not= \emptyset.
$$
Choosing the endpoints of $\tilde I$ from the set in
(\ref{distortioarvio7}) then gives the bound (\ref{kolmiot}).
Note that the estimates hold also for $n=0$, since by (\ref{ext17}) we have $K(z,F) \leq 5/4$ whenever $y \geq 1$.
Hence the proof of Theorem
\ref{alaReed} is complete. \halmos \end{proof}



\section{Exponential of GFF and random homeomorphisms of $\torus$}
\label{se:circlehomeo}

\subsection{Trace of the Gaussian Free Field }

Let us recall that the 2-dimensional Gaussian Free Field (in other
words, the massless free field) $Y$ in the plane has the covariance
$$
\expec Y(x)Y(x')=\log \left(\frac{1}{|x-x'|}\right), \quad x,x'\in
\real^2.
$$
Actually, the definition of this field in the whole plane has to be
done carefully, because of the blowup of the logarithm at infinity.
However, the  definition of the trace $X:=Y_{|\torus}$ on the
unit circle $\torus$  avoids this problem, since it is formally obtained
by requiring (in the convenient complex notation)
 \beqla{eq:tr}
\expec X(z)X(z')=\log \left(\frac{1}{|z-z'|}\right), \quad z,z'\in
\torus .
 \eeq

The above definition needs to be made precise. In order to serve
also readers with less background in non-smooth stochastic fields,
let us first recall the definition of Gaussian random variables with
values in the space of distributions $\distr (\torus ).$ Recall
first that an element in  $F\in \distr (\torus )$ is real-valued if
it takes real values on real-valued test
functions. Identifying $\torus $ with $[0,1 )$ a  real-valued $F$ may be
written as
 $$F =a_0 +\sum_{n=1}^{\infty}\bigl(a_n\cos (2\pi nt)+b_n\sin (2\pi
nt)\bigr),$$
with real coefficients satisfying $|a_n|,|b_n|=O(n^a)$ for
some $a\in\real$. Conversely, every such Fourier series converges in
$\distr (\torus ).$

 Let $(\Omega ,\calF ,\prob )$ stand for a probability space. A map
$X:\Omega\to \distr (\torus)$ is a (real-valued) centered $\distr
(\torus)$-valued Gaussian if for every (real-valued) $\psi\in
\smooth_0(\torus )$ the map
$$
\omega\mapsto \langle X(\omega ),\psi \rangle
$$
is a centered Gaussian on $\Omega$. Here $\langle\cdot
,\cdot\rangle$ refers to the standard distributional duality.
Alternatively, one may define such a random variable by requiring
that a.s.
$$
X(\omega )=A_0(\omega ) +\sum_{n=1}^{\infty}\bigl(A_N(\omega )\cos (2\pi
nt)+B_n(\omega )\sin (2\pi nt)\bigr)
$$
where the $A_n,B_n$ are centered Gaussians satisfying  $\expec
A^2_n,\expec B^2_n =O(n^a)$ for some $a\in\real$. The random
variable $X$ is stationary if and only if the coefficients
$A_0,A_1,\ldots ,B_1,B_2,\ldots$ are independent.

 Due to Gaussianity, the
distribution of $X$ is uniquely determined by the knowledge of the
covariance operator $C_X:\smooth (\torus )\to \distr (\torus ),$
where
 $$ \langle C_X\psi_1,\psi_2 \rangle:=\expec\langle X(\omega
 ),\psi_1
\rangle\langle X(\omega ),\psi_2 \rangle .
 $$
In case the covariance operator has an integral kernel we use the
same symbol for the kernel, and in this case for almost every
$z\in\torus$ one has
$$
(C_X\psi ) (z) =\int_\torus C_X(z,w)\psi (w)m(dw),
$$
where $m$ stands for the normalized Lebesgue measure on $\torus .$
Most of the above definitions and statements carry directly on
$\sdistr(\real )$-valued random variables, but the above knowledge
is enough for our purposes.

The exact definition of  \refeq{eq:tr} is understood in the above sense:
\medskip

\defin{def:X} The trace $X$ of the 2 dimensional GFF on $\torus$
is a centered $\distr (\torus )$-valued Gaussian random variable such that
its covariance operator has the integral kernel
 $$ 
C_X(z,z')=\log\left(\frac{1}{|z-z'|}\right), \quad z,z'\in \torus
.
$$
\edefin
\medskip

 Observe that in the identification $\torus =[0,1)$ the covariance
 of $X$ takes the form
 \beqla{eq:trcov}
C_X(t,u)= \log\left(\frac{1}{2\sin
(\pi|t-u|)}\right)\quad\mbox{for}\;\; t,u\in [0,1).
 \eeq
 The existence of such a field is most easily established by 
 writing down the Fourier expansion:
 \beqla{eq:trsum}
X= \sum_{n=1}^{\infty}\frac{1}{\sqrt{n}}\bigl(A_n\cos (2\pi nt)+B_n\sin
(2\pi nt)\bigr), \quad t\in [0,1),
 \eeq
where all the coefficients $A_n\sim N(0,1)\sim B_n$ ($n\geq 1$) are
independent standard Gaussians. Writing $X$ as $\sum_{n}\frac{1}{\sqrt{n}}(\alpha_nz^n+
\bar\alpha_n\bar z^{n})$ with $|z|=1$ and $\alpha=\hf(A+iB)$ it is readily checked
that it has the stated covariance.
\medskip


What makes the trace $X$ of the $2$ dimensional GFF  particularly  natural for the circle homeomorphisms  are its invariance properties, that $X$   is Möbius invariant {\it modulo constants}. To see this
note that the covariance $C(z,z')=\log(1/|z-z'|)$ satisfies
the transformation rule
$$
C(g(z),g(z'))=C(z,z')+A(z)+B(z'),
$$
where $A$ (resp. $B$) is independent of $z'$ (resp. $z$),
whence the last two terms vanish in integration against mean zero test-functions.
\bigskip
 
It is well-known that with probability one $X(\omega )$ is not an
element in $L^1(\torus ),$ (or a measure on $\torus$), but it just
barely fails to be a function valued field. Namely, if $\varepsilon
>0$ and one considers the $\varepsilon$-smoothened field $(1-\Delta )^{-\varepsilon}X,$ one computes that this
field has a H\"older-continuous covariance, whence its realization
belongs to $C(\torus )$ almost surely. This follows from   the following fundamental result of Dudley that we will use repeatedly below.
\begin{theorem} \label{dudley} Let $(Y_t)_{t\in T}$ be a centered Gaussian field indexed by the set $T$, where $T$
is a compact metric space with distance $d$. Define
the (pseudo)distance $d'$ on $T$ by setting $d'(t_1,t_2)=(\expec
|Y_{t_1}-Y_{t_2}|^2)^{1/2}$ for $t_1,t_2\in T.$  Assume that $d':T\times T\to \real$ is continuous. For $\delta
>0$ denote by $N(\delta )$ the minimal number of balls of
radius $\delta$ in the $d'$-metric needed to cover $T$.
If 
\beqla{eq:dudley}
\int_0^{1}\sqrt{\log N(\delta )}\, d\delta <\infty,
\eeq
then $Y$ has a continuous version, i.e.   almost surely
the map $T\ni t\mapsto Y_t$ is continuous.
\end{theorem}
For a proof we refer to \cite[Thm 1.3.5]{Adler} or \cite[Thm 4, Chapter 15]{Ka2}. 
The second result we will need is the Borell-TIS inequality
(due to C. Borell, or, independently, B. Tsirelson, I. Ibragimov and
V. Sudakov). 
According to the inequality, the tail of the supremum is dominated by a Gaussian tail:
\beqla{eq:Borel-TIS}
\prob (\sup_{t\in T} |Y_t|>u)\leq A\exp (Bu-u^2/2\sigma_T^2),
\eeq
where $\sigma_T:=\max_{t\in T}(\expec Y_t^2)^{1/2},$
and the constants $A$ and $B$ depend on $(T,d')$,
see \cite[Section 2.1]{Adler}.
We shall also need an explicit quantitative version
of this inequality in the special case where $T$ is an interval:
\lem{le:Talagrand}
Let $T=[x_0,x_0+\ell]$, and suppose that the covariance is Lipschitz
continuous with constant $L,$ i.e. $\expec |Y_t-Y_{t'}|^2
\leq L|t-t'|$ for $t,t'\in T.$ Assume also that $Y_{t_0}\equiv 0$ for a $t_0\in T.$ Then
$$
\prob (\sup_{t\in T} |Y_t|>\sqrt{L\ell}u)\leq c(1+u)e^{-u^2/2},
$$
where $c$ is a universal constant. 
\elem
\begin{proof}
The result is essentially due to Samorodnitsky \cite{Sa}
and Talagrand \cite{Ta}. It is a direct consequence of
\cite[Thm 4.1.2]{Adler} since after scaling
it is possible to assume that $L=1=\ell$, and then $\sigma_T
\leq 1$ and $N(\varepsilon )\leq 1/\varepsilon^2.$\halmos
\end{proof}

\subsection{White noise expansion}
\label{subse:white}

The Fourier series expansion  \refeq{eq:trsum} is often not the most suitable representation of $X$ for explicit calculations. Instead, we shall apply a representation that uses
white noise in the upper half plane, due to Bacry and Muzy
\cite{BaMu}. The white noise representation is very convenient since it allows one  to consider correlation between different scales both
in the stochastic side and on $\torus$ in a flexible and geometrically transparent manner. Moreover, as we define the exponential
of the field $X$ in the next subsection we are then able to refer to
known results in \cite{BaMu} and elsewhere.

To commence with, let $\lambda$ stand for the hyperbolic area measure in the upper half plane $\BbbH$,
$$
\lambda (dxdy )=\frac{dxdy}{y^ 2}.
$$
Denote by $w$ a white noise in $\BbbH$ with respect to measure
$\lambda $. More precisely,   $w$ is a centered Gaussian process indexed by
Borel sets  $A\in \CB_f (\BbbH)$, where
$$
\CB_f (\BbbH):=\{ A\subset \BbbH \; \mbox{Borel} \; |\;
\lambda (A)<\infty \;\mbox{and} \; \sup_{(x,y),(x',y')\in A}|x'-x| <\infty \},
$$
i.e. Borel sets of finite hyperbolic area and finite width, and with
the covariance structure
$$
\expec \big( w(A_1)w(A_2 )\big ) =\lambda (A_1\cap A_2),\quad
A_1,A_2\in \CB_f (\BbbH).
$$
We shall need a periodic version of $w$, which can be
identified with a white noise on $\torus\times\real_+ .$  Thus,
define $W$ as the centered Gaussian process, also indexed
by $\CB_f (\BbbH )$, and with covariance
$$
\expec \big(  W(A_1) W(A_2 )\big ) =\lambda \Bigl(A_1\cap \bigcup_{n\in\IZ} (A_2+n)\Bigr).
$$

\vskip 5mm

\begin{figure}[h]

\setlength{\figheight}{.25\textheight}

\begin{center}
\includegraphics[height=\figheight]{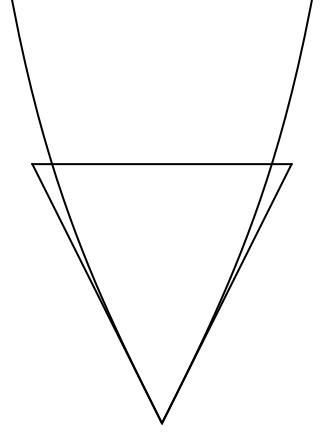}
\caption{White noise dependence of the fields $H(x)$ and $V(x)$}
\vspace{-1.1\figheight}
$H+x$\\
\vspace{.3\figheight}
$V+x$\\
\vspace{.45\figheight}
$x$\\
\end{center}

\end{figure}

\vskip 10mm

We will represent the trace $X$ using  the following random field $H(x)$.  
Consider the wedge shaped
region
$$
H:= \{ (x,y)\in\BbbH \;: \; -1/2<x< 1/2,\;\; y>\frac{2}{\pi}\tan
(|\pi x|)\} 
$$
and formally set
$$
H(x):=W(H+x), \quad \quad x \in \R/\IZ,
$$
see Fig. 1.
The reader should think about the $y$ axis as parametrizing the spatial scale. Roughly,
the white noise at level $y$ contributes to $H(x)$ in that spatial scale.

To define $H$
rigorously we introduce a short distance cutoff parameter $\varepsilon>0$ and, given
any $A\in \CB_f (\BbbH)$, let  $A_\varepsilon := A\cap \{ y>\varepsilon\}$. Then set
\beqla{eq:truncation}
H_\varepsilon (x):=W(H_\varepsilon+x).
\eeq

According to  Dudley's Theorem  \ref{dudley}  one may pick a version of the
white noise $W$ in such a way that the map
$$
(0,1)\times\real \supset (\varepsilon ,x)\mapsto 
H_\varepsilon (x)
$$
is continuous. In the limit $\varepsilon\to 0^+$ we nicely recover
$X$: 

\lem{le:whitetrace} One may assume that the version of the
white noise is chosen so that for any $\varepsilon >0$  the map
$x\mapsto H_\varepsilon (x)$ is continuous, and as
$\varepsilon\to 0^+$  it converges in $\distr (\torus)$ to a random field
$H$. Moreover  
$$
H\sim  X+G,
$$
 where $G\sim N(0,2\log
2)$ is a (scalar) Gaussian factor, independent of $X$.
 \elem
 
\begin{proof}
Observe first that we may compute formally (as $ H
(\cdot )$ is not well-defined pointwise) for $t\in (0,1)$
\beqla{eq:formal} \expec H(0)H
(t)\nonumber = \nonumber \lambda (H\cap (H+t))+\lambda (H\cap
(H+t-1)). \eeq The first term in the right hand side can be computed
as follows:
 \beqla{eq:1term}
 \lambda (H\cap (H+t))&=&2\int_{t/2}^{1/2}\left(
 \int_{(2/\pi )\tan (\pi x)}^\infty
 \frac{dy}{y^2} \right )dx = \pi\int_{t/2}^{1/2}\cot (\pi x)\, dx
 \nonumber\\
&=& \log\frac{1}{\sin(\pi t/2)}.\nonumber \eeq Hence we obtain by
symmetry
 \beqla{eq:symmetry}  \expec H (0)
H (t) &=& \log\Bigl(\frac{1}{\sin(\pi t/2)}\Bigr)+
\log\Bigl(\frac{1}{\sin(\pi (1-t)/2)}\Bigr) \\ &=&\nonumber 2\log
2+\ \log\Bigl(\frac{1}{2\sin(\pi t)}\Bigr). \eeq 
The stated relation
between $X$ and $H$ follows immediately from this as soon as we
prove the rest of the theorem. Observe that the covariance of the
smooth field $H_\varepsilon (\cdot )$ on 
$\torus$ converges to the above pointwise for $t\neq 0$. A
computation shows that for any $\delta
>0$ the  covariance of the field
$$
[0,1]\times [0,1) \supset (\varepsilon, x) \mapsto (1-\Delta
)^{-\delta } H_\varepsilon (x ):=H_{\varepsilon,\delta} (x )
$$
(at $\varepsilon  =0$ one applies the covariance computed in
\refeq{eq:symmetry}) is H\"older-continuous in the compact set
$[0,1]\times\torus$, whence Dudley's theorem yields the existence of a
continuous version on that set, especially $ H_{\varepsilon,\delta}
(\cdot ) \to H_{0,\delta} (\cdot)$ in $C(\torus )$, hence in $\distr
(\torus )$. By applying $(1-\Delta )^{\delta }$ on both sides we
obtain the stated convergence. Especially, we see that the
convergence takes place in any of the Zygmund spaces $C^{-\delta
}(\torus )$, with $\delta >0. $ \halmos

\end{proof}

The logarithmic singularity in the  covariance of $H(x)$ is produced
by the asymptotic shape of the region $H$ near the real axis.
 It will be often convenient to work with the following auxiliary field, which is
geometrically slightly easier to tackle while  for small
scales it does not distinguish between $w$ and its periodic
counterpart $ W$. Thus,  consider this time the
triangular set
 \beqla{eq:Vregion}
V:=\{ (x,y)\in\BbbH \;: \; -1/4<x< 1/4,\;\; 2|x|<y< 1/2\} .
\eeq
and let $ V_\varepsilon (x)=W(V_\varepsilon+x)$ (see Fig 1.).
The existence of the limit $V(x):= \lim_{\varepsilon\to 0^+}V_\varepsilon(\cdot )$ is
established just like for $H$ and we get the covariance 
 \beqla{eq:covV}
\expec V(x) V
(x')=\log\Bigl(\frac{1}{2|x-x'|}\Bigr)+{2 |x-x'|}-1\ \eeq
for $|x-x'|\leq 1/2$ (while for  $|x-x'|> 1/2$ the periodicity must be taken into account).

Since the regions $H$ and $V$ have the same slope at the real axis
the difference
$H(\cdot )-V(\cdot )$ is a quite regular field:

\lem{le:difference} Denote $\xi:=\sup_{x\in [0,1),\varepsilon\in
(0,1/2]}| V_\varepsilon (x)- H_\varepsilon
(x)|.$ Then almost surely $\xi<\infty$. Moreover, $\expec
\exp(a\xi)<\infty$ for all $a>0.$
 \elem
 
\begin{proof}
We may write for $\varepsilon \in [0,1/2]$
$$
 V_\varepsilon (x)-H_\varepsilon (x) =
 T_\varepsilon (x)-  G(x),
$$
where $T_\varepsilon(x)$ and $G(x)$ are constructed as $V_\varepsilon(x)$ out of the sets
 $$G:=H\cap \{ y\geq 1/2\},\ \ \,T:=V\setminus (H\cap \{ y<
1/2\}).
$$ 

Observe first that $ G (x)$ is independent of
$\varepsilon$ and it clearly has a Lipschitz covariance in $x$, so by
Dudley's theorem and (\ref{eq:Borel-TIS})  almost surely the
map $ G (\cdot )\in C(\torus)$ and, moreover, the tail
of $\| G (\cdot )\|_{C(\torus)}$ is dominated by a
Gaussian, whence it's  exponential moments are  finite.

In a similar manner, the exponential integrability of $\sup_{x\in
[0,1),\varepsilon\in [0,1]}|T_\varepsilon (x)|$ is
deduced from Dudley's theorem and \refeq{eq:Borel-TIS}  as soon as we verify that there is an
exponent $\alpha >0 $ such that for any $|x-x'|\leq 1/2$ we have
 \beqla{eq:holder}
 \expec | T_\varepsilon (x)- 
T_{\varepsilon'}(x')|^2\leq
c(|x-x'|+|\varepsilon-\varepsilon'|)^\alpha . \eeq In order to
verify this it is enough to change one variable at a  time. Observe
first that if $1>\varepsilon >\varepsilon'\geq 0$
\beq
 \expec | T_\varepsilon (x)- 
T_{\varepsilon'}(x)|^2 =\lambda \bigl(T\cap (\varepsilon'<y<\varepsilon )\bigr) \leq
\int_{\varepsilon'}^\varepsilon cx^3\leq
c'|\varepsilon'-\varepsilon|,\nonumber
 \eeq
 where we  applied the inequality
$0\leq (t/2)-(1/\pi)\arctan ( \pi t/2)\leq 2t^3.$  

Next we estimate
the dependence on $x$. Denote $z:=|x-x'|\leq 1/2$. We note that for
any $y_0\in (0,1/2) $ the linear measure of the  intersection $\{
 y=y_0\}\cap (T\Delta (T+z))$ is bounded by $\min (2z,4y_0^3)$. Hence, by
 the definition of $T_\varepsilon$ and the fact that for $z=|x-x'|\leq 1/2$
the periodicity of $W$ has no effect on estimating $T$, we obtain 
 \beq
 \expec | T_\varepsilon (x)- 
T_{\varepsilon}(x')|^2 &\leq& \expec | T_0(x)-
T_0(x')|^2=\lambda (T\Delta(T+z))\nonumber\\
&\leq &2z\int_{z^{1/3}}^{1/2}\frac{dy}{y^2}
+\int_0^{z^{1/3}}\frac{4y^3}{y^2}\leq cz^{2/3},
 \nonumber
 \eeq
 which finishes the proof of the lemma.
\halmos
\end{proof}
\bigskip

\subsection{Exponential of $X$ and the random homeomorphism $h$.}\label{subse:exponential}

We are now ready to define the exponential of the
free field discussed in the Introduction and use it to define  the random circle homeomorphisms.

By stationarity, the covariance
$$
\gamma_{H}(\varepsilon):={\rm Cov \,}(
H_\varepsilon (x))=\expec | H_\varepsilon (x)|^2
$$
is independent of $x$, as is the quantity
$\gamma_{ V}(\varepsilon)$
defined analoguously. Fix
$\beta >0$ (this parameter could be thought as an "inverse
temperature"). Directly from definitions, for any $x$ and for any
bounded Borel-function $g$ on  [0,1) the processes
\beqla{eq:martingales}
 && \varepsilon\mapsto \exp \bigl(\beta H_\varepsilon
 (x)-(\beta^2/2)\gamma_{ H}({\varepsilon})\bigr)\qquad \mbox{and}\\
 && \varepsilon\mapsto \int_0^1\exp \bigl(\beta H_\varepsilon
 (u)-(\beta^2/2)\gamma_{ H}(\varepsilon)\bigr)g(u)\, du
\eeq are $L^1$-martingales with respect to {\it decreasing}
$\varepsilon\in (0,1/2],$ whence they converge almost surely.
Especially, the $L^1$-norm stays bounded and the Fourier-coefficients
of the density $\exp \bigl(\beta H_\varepsilon
 (x)-(\beta^2/2)\gamma_{ H}(\varepsilon)\bigr)$ 
converge as $\varepsilon\to 0^+$. 

Now comparing these expressions with (\ref{f1}) and Lemma \ref{le:whitetrace} we are led to the exact definition of our desired exponential 
$$"d\tau= e^{\beta X(z)}dz".$$ 
Indeed,   by the  weak$\hbox{}^*$-compactness
of the set of bounded positive measures we have the existence of  
the almost sure limit measure\footnote{Observe that the limit measure is weak$\hbox{}^*$-measurable in the sense that
for any $f\in C(\torus )$ the integral $\int_\torus f(t)\tau (dt)$ is a well-defined random variable. In this paper all our random measures on $\torus$ are  measurable 
(i.e. they are measure--valued random variables) 
in this sense. A simple limiting argument then shows  that
e.g. $\tau (I)$ is a random variable for any interval $I\subset\torus$.}
\beqla{eq:limitmeasure}\mbox{a.s.}\qquad \lim_{\varepsilon\to 0^+}
e^{ \bigl[\beta H_\varepsilon
 (x)-(\beta^2/2)\gamma_{ H}(\varepsilon)\bigr]} e^{-\beta \,G} \,dx/2^{\beta^2}  =:\tau (dx) \quad
 \rm{w}^*\;\;\mbox{in}\;\;\CM(\torus ),                        \eeq
 where $\CM (\torus)$ stands for bounded Borel measures on
 $\torus$
 and $G\sim N(0,2\log 2)$ is a
 Gaussian (scalar) random variable. 
 \medskip
 
In a similar manner one deduces the existence of the almost sure
limit \beqla{eq:limitmeasure2} && \lim_{\varepsilon\to 0^+} \exp
\bigl(\beta V_\varepsilon
 (x)-(\beta^2/2)\gamma_{ V}(\varepsilon)\bigr)dx
 \stackrel{ \rm{w}^*}{=}:\nu(dx)
                      \eeq
 Lemma
\ref{le:difference} and stationarity yield
immediately 

\lem{le:comparision} There are versions of $\tau$
and $\nu$ on a common probability space, together with an almost
surely finite and positive random variable $G_1$, with $\expec
G_1^a<\infty$ for all $a\in\real,$ so that for all Borel sets 
$B$  one has
$$
\frac{1}{G_1} \, \tau (B)\leq \nu(B)\leq {G_1}\, \tau (B).
$$
\elem 

Observe that the random variable $G_1$ is independent of the set $B$. Thus, the measures are a.s. comparable.
\medskip

Limit measures of above type, i.e. measures that are obtained as
martingale limits of products (discrete, or continuous as in our
case) of exponentials of independent Gaussian fields have been
extensively studied in the literature. The study
 of "multiplicative chaos " starts with Kolmogorov, various versions of multiplicative cascade models
 were advocated by Mandelbrot \cite{Mandel} and others, and Kahane made fundamental contributions
 to the rigorous mathematical theory, see \cite{Ka1},\cite{Ka3}, \cite{KaPe}.  We shall make use of these works, and  \cite{BaMu},
\cite{RoVa} 
in particular, which study in detail random measures
closely  related to our  $\nu$. 
 
 For us  the key points in constructing and understanding  the random circle homeomorphism are the following properties of the measure $\tau$ and its variant $\nu$.  

\thm{th:fact1} Assume that
$\beta <\sqrt{2}.$ 

\noindent {\bf (i)}\quad There are $a=a(\beta )>0$ and an a.s. finite 
random constant $c=c (\omega ,\beta )$ 
such that for all subintervals $I\subset [0,1)$ it holds
$$
0<\tau (I)\leq c(\omega,\beta)|I|^a.
$$ 
Especially, $\tau$ is non-atomic.

\smallskip

\noindent{\bf (ii)}\quad For any subinterval $I\subset [0,1)$ the
measure $\tau$ satisfies
\begin{equation}
\tau(I)\in L^p(\omega), \ \ \ p\in(-\infty,2/\beta^2).
\label{lp1}
\end{equation}
Moreover, if $p\in (1,2/\beta^2)$, then
\begin{equation}
\E \,\tau(I)^p<c(\beta,p)|I|^{\zeta_p},\qquad \mbox{with}\;\; \zeta_p >1\label{lp}.
\end{equation}

\noindent{\bf (iii)}\quad One can replace $\tau$ by the measure $\nu$ in the
statements (i) and (ii). \ethm

\begin{proof}
 We shall make use of one more auxiliary field, which (together with its exponential)
  is described in
 detail in \cite{BaMu}\footnote{$U_0$ corresponds to the simple case of log-normal MRM,
 see  \cite[p. 462, (28)]{BaMu}, and $T=1$ in \cite[p. 455, (15)]{BaMu}.}. Define
 $$
U:=\{ (x,y)\in\BbbH \;: \; -1/2<x< 1/2,\;\; 2|x|<y\} .
 $$
 and  for $x\in\BbbR$, let $U(x)=w(U+x)$. Here note in particular, that $w$ is the nonperiodic white noise.
 
 The covariance of $U(\cdot
)$ is easily computed (see \cite[(25), p. 458]{BaMu}), and we obtain
 \beqla{eq:covH}
&&\expec U(x) U
(x')=\log\Bigl(\frac{1}{\min  (y,1)}\Bigr)\quad \mbox{where}\;\;
y:=|x-x'|.
 \eeq
As before define the cutoff field  $U_\varepsilon (x)=w(U_\varepsilon+x)$. Then
 $U_\varepsilon$ is  (locally) very close
to our field $V_\varepsilon(\cdot )$. 
Indeed, let $I$ be an interval of length $|I|=\hf$. Then
$V(\cdot)|_I$ is equal in law with $w(\cdot+V)|_I$ since
the periodicity of the white noise $W$ will not enter.
Thus we may realize $U_\varepsilon|_I$ and $V_\varepsilon|_I$
 for $\varepsilon \in (0, 1/2)$ 
 in the same
probability space so that 
$$
U_\varepsilon-V_\varepsilon:=Z= w(x+U\cap \{y>1/2\}).
$$
We may again apply Dudley's theorem and
 eq.  \refeq{eq:Borel-TIS}  to the random
variable
  \beqla{eq:difference2}
 \xi_1:=\sup_{x\in I, \, \varepsilon\in
(0,1/2]}
|V_\varepsilon (x)-U_\varepsilon (x)| < \infty \quad \mbox{almost surely.}
 \eeq
  Moreover, $\expec \exp(a\xi_1)<\infty$ for
all $a>0.$  By denoting $G_2:= \exp(a\xi_1)$ we thus have 
an analogue of Lemma \ref{le:comparision}
  \beqla{eq:difference3}
\frac{1}{G_2} \, \tau (B)\leq \nu(B)\leq {G_2}\, \tau (B),
 \eeq
 for all $B\subset I$, and the auxiliary variable $G_2$ satisfies
$\expec G_2^p<\infty$ for all $p\in\real$.
As an aside, note that we cannot have \refeq{eq:difference3} for the full interval $I = [0,1]$, as $V$ is $1$-periodic while $U$ is not.

In a similar manner as for the measures $\tau$ and $\nu$
one deduces the existence of the almost sure
limit
 \beqla{eq:limitmeasure3} 
   \lim_{\varepsilon\to 0^+} \exp \bigl(\beta U_\varepsilon
 (x)-(\beta^2/2)\gamma_{U}(\varepsilon)\bigr)dx =:\eta (dx), \quad
                      \eeq
where the  limit takes place locally weak$\hbox{}^*$ on the space
of locally finite Borel-measures on the real-axis. 

Now for proving the theorem, by (\ref{eq:difference2}) and Lemma \ref{le:difference} it is enough to check
the corresponding claims (i) and (ii) for the random measure $\eta $, 
as one may clearly assume that $|I|\leq 1/2$.
We start with claim  (ii), which  in the  case of positive moments $p>0$  is
due to Kahane (see \cite{KaPe},\cite{Ka1}). Bacry and Muzy \cite[Appendix D]{BaMu}
give a nice proof  by adapting the argument  of Kahane and Peyriere
\cite{KaPe} (who considered a cascade model) to cover the measure $\eta$. 

Finiteness of negative moments is announced in \cite[Prop. 3.5]{RoVa},
where it is stated that the argument given by Molchan  \cite{Mo}
in the case of the cascade model carries through. For the readers convenience,
we  include the details for the negative moments in an  Appendix.

Fact \refeq{lp} for $\eta$ is \cite[Theorem 4]{BaMu} where
it is observed that one may take $\zeta_p=p-\beta^2(p^2-p)/2$.
In order to treat (i), 
choose $p\in (1,2/\beta^2)$ and let $a>0$ be so small that
$b:=\zeta_p-pa>1.$
Chebychev's inequality in combination with \refeq{lp} yields that
$\prob (\eta (I)>|I|^a )\lesssim |I|^b$. In particular, $ \sum_{I}\prob (\eta (I)>|I|^a )<\infty$, where one sums over dyadic subintervals of $[0,1)$. The same holds
true if one sums over the same dyadic subintervals shifted by their half-length. This observation in combination with  the Borel-Cantelli lemma yields the desired upper estimate for the measure $\eta $. This immediately implies that $\eta$
is non-atomic.

Finally, in order to sketch a proof of the non-degeneracy of $\eta$ over any subinterval, we partition the upper half plane into vertical strips and define 
\begin{equation}
U (x,j):= w((U+x)\cap
\{1/(j+1) <y< 1/j\}).
\label{Veps}
\end{equation}
Let then
$$  f_j(x) = \exp \bigl(\beta U(x,j)-(\beta^2/2)\gamma_{U}(j)\bigr),$$ where  $  \gamma_{U}(j) = {\rm Cov \,}\bigl(U(x,j)\bigr)$. 
We may now write $\eta$ as the a.s.
limit
$$
\eta (dx)=
\mbox{w$\hbox{}^*$-}\lim_{k\to\infty}\left(\prod_{j=0}^kf_j(x,\omega)\right)dx,
$$ where the densities $f_j(x,\omega)$  are independent and
a.s. bounded from below by a positive constant. Moreover, $\expec f_j(x)=1$
for each $x,j$. Let $I$ be a dyadic
subinterval and denote $Y_k(\omega):=\int_I\Bigl(\prod_{j=1}^kf_j(x,\omega)\Bigr)\, dx.$ By Kolmogorov's 0-1 law, probability for $\lim_{k\to\infty} Y_k=0$ is either zero or one. The first alternative can be ruled out by observing that $\expec Y_k=|I|$ for all $k$
and that $(Y_k)_{k\geq 1}$ in an $L^p$-martingale
with $p>1$, according to fact (ii).
Let us finally remark that the non-degeneracy and non-atomic nature  for $\eta$ can also be found in \cite{Ka1} and \cite{Ka3}, see also \cite[Theorems 1 and 2]{BaMu}.
\halmos
\end{proof}

Note that the exact scaling law of the measure $\eta$ we used in  the above proof  is given in
 \cite[Thm. 4]{BaMu}. Indeed, for any $\varepsilon, \lambda\in (0,1)$ one has the equivalence of
laws
 $$U_{\varepsilon\lambda}(\lambda \cdot)|_{[0,1]} \sim G_\lambda + U_\varepsilon |_{[0,1]} 
 $$
where $G_\lambda\sim N(0,\log (1/\lambda ) )$ is a Gaussian independent of
$U.$  Therefore,  one has the equivalence of laws for measures on $[0,1]$:
 \beqla{eq:scaling} \eta (\lambda\cdot)\; \sim \;
\lambda e^{\beta G_\lambda + \log(\lambda) \beta^2/2} \, \eta \eeq 
and hence  scale invariance of the ratios
 \beqla{eq:scaling2} \frac{\eta ([\lambda x,\lambda y])}{\eta ([\lambda a,\lambda b])}Ê\; \sim \;
 \frac{\eta ([ x, y])}{\eta ([ a, b])}. \eeq 
In turn, the exact
scaling law of $\tau$ is best described in terms of M\"obius
transformations of the circle. We do not state it as we
do not need it later on. 
\medskip

To finish this section we are now able define our circle homeomorphism $h$.
\smallskip

\defin{def:randomhomeo}
Assume that $\beta^2 <2.$ The random homeomorphism
$\phi:\torus\to\torus$ is obtained by setting 
 \beqla{randomhomeo1}
  \phi(e^{2\pi i x}) = e^{2\pi i h(x)},
  \eeq
   where we let
 \beqla{randomhomeo2}
h(x)= h_\beta(x)= \tau([0,x])/\tau([0,1]) \quad  \; for  \,\; x\in [0,1),
\eeq 
and extend periodically over  $\R$.
\edefin 
\smallskip

\begin{remark}  Theorem \ref{th:fact1}  (i) and (ii)  precisely contain what is needed to ensure that  $h$ is 
a H\"older continuous homeomorphism for $\beta^2< 2$.
As an aside, let us note
that  defining $\tau_\varepsilon$ as in  the LHS of eq. \refeq{eq:limitmeasure}
the limit $\lim_{\varepsilon\to 0}\tau_\varepsilon=0$ for  $\beta^2\geq 2$.
However, it is a natural conjecture that letting 
$h_\varepsilon$ to be given by  \refeq{randomhomeo2}
with $\tau$ replaced by $\tau_\varepsilon$, the  limit  for  $h_\varepsilon$ exists in a suitable sense as
$\varepsilon\to 0^+$  also for $\beta^2\geq 2$. Indeed, the normalized measure 
in eq.  \refeq{randomhomeo2} appears in the physics literature as
the Gibbs measure of a Random Energy model for logarithmically 
correlated energies \cite{leD}, \cite{bf}, \cite{bf1} and the $\beta^2> 2$ corresponds
to a low temperature "spin glass" phase. However,
we don't expect the limiting $h$ to be  continuous if $\beta^2> 2$. 
\end{remark}
\smallskip

\begin{question}  \  $\hskip1pt$ Is $\beta \to h_\beta(x)$ almost surely continuous ?
\end{question}
\smallskip

\section{Probabilistic estimates for Lehto integrals}
\label{se:probability}

\subsection{Notation and statement of the main estimate}
\label{se:statement}

We will now set to study the Lehto integral of
eq. (\ref{eq:4.7}) for the random homeomorphism constructed in the previous
section. As explained in Section
\ref{subse:extension}, it suffices to work
in the infinite strip $S=\mathbb R\times [0,2]$ where the extension $F$
  of the random 
homeomorphism $h$ is non-trivial. We use  the bound (\ref{kolmiot})  
for the (random) pointwise  distortion $K =K(z,F)$ of this extension, and  hence it  turns out convenient to define
$K_\tau$ in the upper half plane by setting
 \beqla{Koot}   K_\tau(z) :=   K_\tau(I) \quad \mbox{whenever }  z \in C_I.
  \eeq 
A lower bound for the Lehto integral
 (\ref{eq:4.7})
is then obtained by replacing $K$ there by $ K_\tau$. We similarly define $ K_\nu(z)$ for $z \in \IH$, via the modified Beurling - Ahlfors extension of the periodic homeomorphism  defined by the measure $\nu$.

It turns out that we only need  to control Lehto integrals centered at real 
axis and with some (arbitrarily small, but fixed) outer radius. For this
purpose
fix  (large) $p\in\BbbN$  and choose $\rho=2^{-p}$, where final choice of $p$
will be done in Subsection \ref{subse:lln} below.

Our main probabilistic
estimate is the following result. 
\thm{mainestimate}  Let $w_0\in\real$ and let $\beta<\sqrt{2}$. Then there exists $b>0$ and $\rho_0 >0$  together with $\delta(\rho) >0$ such that for
positive $\rho<\rho_0$ and $\delta<\delta(\rho)$ 
the Lehto integral satisfies the estimate
\begin{equation}
\prob\bigl(L_{{K}_\nu }(w_0,\rho^N,2\rho)<N\delta)\bigr)\leq
 \rho^{(1+b )N}.
 \label{prob}
\end{equation}
\ethm

Observe that
the estimates in  the Theorem  are in terms of 
${K}_\nu$ instead of ${K}_\tau$, which is the
majorant for the distortion of the extension of the actual
homeomorphism. However, this discrepancy will easily be taken care
later on in the proof of Theorem \ref{th:4.1} using the bounds in Lemma \ref{le:difference}. The proof of  of the Theorem will occupy  most of the present  section, i.e. Subsections \ref{subse:correlation}--\ref{subse:proposition} below. Finally, we consider the almost sure integrability of
the distortion
in Subsection \ref{subse:L1}.

\bigskip

We next fix the notation that will be used for the rest of the present
section, and explain the philosophy behind part (i) of the theorem.
Given $w_0$ we may choose the dyadic intervals in Theorem 2.6.
as $w_0+I$. Then, 
 by stationarity we may assume that $w_0=0$. Let $S_r$
denote the circle of radius $r>0$ with center at the origin. 
Define (with slight abuse)  for $r\leq 2\rho$ 
\beqla{Kr} K_\nu(r):=\sum_{I:C_I\cap
S_r\neq\emptyset}|I|{K}_\nu (I) 
\eeq
 and observe that
\beqla{eq:vertaa} L_{K_\nu}(0, \rho^N,2\rho )\geq c\sum_{n=1}^NM_n,
\eeq where 
\beqla{Mn} M_n=\int_{\rho^n}^{2\rho^n}{dr\over K_\nu(r)}.
\eeq Thus, in order to prove part (i) of the Theorem it is enough to
verify for $\beta <\sqrt{2}$ that for small enough $\rho>0$ and $0<\delta <\delta (\rho)$
one has
\beqla{eq:mainestimate}
\prob (\sum_{n=1}^N M_n<N\delta)\leq 
\rho^{(1+b)N} . 
\eeq

If the summands $M_j$ in \refeq{eq:mainestimate} were independent,
the estimate would  follow easily from basic large deviation
estimates. However, they are far from being independent.
Nevertheless, by the geometry of the setup in the white noise upper half plane,
one expects that there is some kind of exponential decay of dependence,
 but due to the complicated structure of the Lehto integrals we need to
  go through a non-trivial technical analysis in order to be able to get
  hold on the exponential decay.

\bigskip

\subsection{Correlation structure of the $M_j$:s}
\label{subse:correlation}

 In this Section we will study how the random
variables $M_n$ are correlated with each other.
As one can easily gather from
 the representation  of the field $\nu$ in terms of the white
noise, all of the variables $M_n$  with $n=1,2,\ldots$ are correlated with each other. 
Our basic strategy is to estimate $M_n$ from below by the quantity 
$$
M'_n=m_ns_n\sigma_n
$$
(see \refeq{Mnfinal} below), where the random variables $m_n$ depend only the white
noise on the scale $\sim \rho^n$ and form an independent set. The variables $s_n$ will provide an estimate of {\sl upscale correlations}, i.e.  the dependence of $M'_n$  on the white noise over the larger spatial scales $\{ |x|\gtrsim\rho^{n-1}\}$). In turn, the variables $\sigma_n$ measure the {\sl downscale correlations} that corresponds to the dependence of $M'_n$ on white noise over $\{ |x|\lesssim\rho^{n+1} \}$. It turns out that the downscale correlations are harder to estimate. 

 We start with the upscale
correlations and introduce some terminology. For a Borel-measurable
$S\subset \BbbH$ let $\CB_S$ be the $\sigma$-algebra generated by  the randoms variables
$W(A)$, where $A$ runs over  Borel-measurable subsets  $A\subset S$. We will call a  $\CB_S$
measurable random variable for short $S$ measurable.
 Let
$$V_I:=\cup_{x\in I}(V+x)
$$ 
where we recall $V$ is given by (\ref{eq:Vregion}). Then $\nu(I)/\nu(J)$ is $V_{I\cup J}$
measurable and  by (\ref{Kbound}) we see that $K_{\nu}(I)$ is $V_{j( I)}$
measurable (recall that $j(I)$ denotes the union of $I$ with
its neighboring dyadic intervals).  From (\ref{Kr}) we deduce that $M_n$ is $V_{B_n}$
measurable where $B_n:=B(0,4\rho^n)$. Indeed, the Whitney cubes
$C_I$ that intersect the annulus $A_n:=B(0,2\rho^n)\setminus
B(0,\rho^n)$ have $I\subset B(0,2\rho^n)$ and thus $j(I)\subset
B(0,4\rho^n)$.

We now decompose $V(\cdot )|_{B_n}$ to scales using the white
noise. Denote in general for $0\leq \varepsilon <\varepsilon',$
\begin{equation}
V (x,\varepsilon, \varepsilon'):= W((V+x)\cap
\{\varepsilon <y< \varepsilon'\}).
\label{Veps}
\end{equation}
Set
for $n\geq 1$
\begin{equation}
\psi_n(x)=V(x,0, \rho^{n-\hf})
\label{phin}
\end{equation}
and for $k\geq 0$
\begin{equation}
\zeta_k(x)=V(x,\rho^{k+\hf}, \rho^{k-\hf}).
\label{zk}
\end{equation}
Denoting
\begin{equation}
\Lambda_n=\{z\in\BbbH: y\leq \rho^{n-\hf}\} \label{lambdan}
\end{equation}
we see that in any open set  $U$  the field $\psi_n$ is
$(\bigcup_{y\in U}V_y)\cap\Lambda_n$ measurable. In a similar way,  $\zeta_k(x)$ is
$V_x\cap(\Lambda_{k}\setminus\Lambda_{k+1})$ measurable and since
these regions are disjoint
 the field $V$ decomposes to a sum of independent fields
\begin{equation}
V=\psi_n+\sum_{k=0}^{n-1}\zeta_k:=\psi_n+z_n.
\label{decomp}
\end{equation}

Let $\nu_n$ be the measure defined as $\nu$ but with  $V$ 
replaced by $\psi_n$. Inserting the second  decomposition in (\ref{decomp}) to
the measure $\nu$ we have, for any $I,J\subset B_n$
\begin{equation}
{\nu(I)\over\nu(J)}\leq \  {\nu_n(I)\over\nu_n(J)}\cdot{\sup_{x\in B_n}e^{\beta z_n(x)}\over
\inf_{x\in B_n}e^{\beta z_n(x)}}.
\label{mumu}
\end{equation}
The first decomposition in  (\ref{decomp}) then gives 
\begin{equation}
{\sup_{x\in B_n}e^{\beta z_n(x)}\over
\inf_{x\in B_n}e^{\beta z_n(x)}}\leq 
e^{\sum_{k=0}^{n-1}t_{n,k}}
:=s_n^{-1}
\label{sn}
\end{equation}
where
\begin{equation}
t_{n,k}:=\log{\sup_{x\in B_n}e^{\beta \zeta_k(x)}\over
\inf_{x\in B_n}e^{\beta \zeta_k(x)}}.
\label{tn}
\end{equation}
Thus if we let
\begin{equation}
\CM_n=\int_{\rho^n}^{2\rho^n}{dr\over K_{\nu_n}(r)} \label{CMn}
\end{equation}
we arrive to the following lower bound for $M_n$:
\begin{equation}
M_n\geq\CM_ns_n.
\label{decoup}
\end{equation}
This is the desired decoupling upscale. 
Note that the fields $\zeta_k$ become more regular
as $k$ decreases. This will lead to the following Proposition:

\vs{3mm}

\prop{propo1} The random variables $t_{n,k}$ satisfy 
\begin{equation}
\prob (t_{n,k}> u \rho^{(n-k)/2-1/4})\leq
 ce^{-u^2/c}.
 \quad \quad k=0, \dots , n-1,
\label{tauest}
\end{equation}
where $c$ is independent on $\rho$, $n$ and $k$. 
Moreover, $t_{n,k}$ and $t_{n,'k'}$ are independent if $k\neq k'$.
\eprop

The proof of this proposition is postponed to Subsection
\ref{subse:proposition} below.

\vs{3mm}

The decoupling downscale is done to the random variables $\CM_n$ in
(\ref{CMn}). Obviously $\CM_n$ and $\CM_m$ are dependent. However,
as in (\ref{Kr}), most of the terms $K_{n,I}:={K}_{\nu_n}(I)$   are
independent of $\CM_m$ if $m>n$. The few which are not we will
process further in a moment.

So let us first look at the dependence of the $K_{n,I}$ on the white
noise. For $U\subset \BbbR$ set $V^n_U:=V_{U}\cap \Lambda_n$. Then
$K_{n,I}$ is $V^n_{j( I)}$ measurable and $\CM_m$ is $V^m_{B_m}$
measurable. Some drawing will convince the reader that if ${\rm
dist}(j( I), 0)$  is not too small  $K_{n,I}$ and $\CM_m$  are
independent for $m>n$. Indeed, consider the ball
$B'_n=B(0,2\rho^{n+\hf})$ so that $B_{n+1}\subset B'_n \subset
B_n$. The regions $V^n_{B_n\setminus B'_n}$ are disjoint (see Figure
2).
 Thus the $\sigma$-algebras $\CB_{V^n_{B_n}\setminus  V^n_{B'_n}}$ are independent
 from each other for $n=1,2,\ldots $.

\vskip 5mm

\begin{figure}[h]

\setlength{\figheight}{.28\textheight}

\begin{center}
\includegraphics[height=5.0cm]{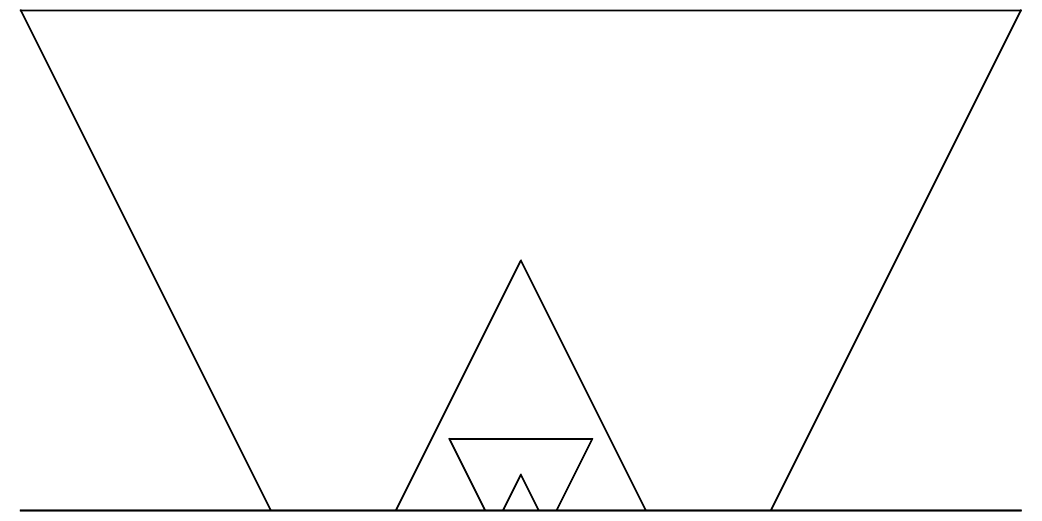}
\caption{A schematic picture of the regions  $V^n_{B_n\setminus B'_n}:=v_n$, where the $m_n$ are measurable}
\vspace{-.8\figheight}
$v_{n-1}$\\
\vspace{-2ex}
\vspace{.43\figheight}
\begin{footnotesize}
$v_{n}$\\
\end{footnotesize}
\end{center}
\end{figure}

\vskip 20mm

Let $\CI_n$ be the set of $I\in\CD$ such that the Whitney cube $C_I$
intersects the annulus $A(0,\rho^n,2\rho^n)$ and $j(I)\cap
B'_n\neq\emptyset$ (some drawing shows such $I\in\CD_{np+i}$ for $i=0,\pm1$) . Moreover, for each fixed $r\in (\rho^n,2\rho^n)$ let  $\CI_n(r)$ consist of those intervals $I$ for which
 $C_I\cap S_r\neq \emptyset$ and $ j(
I)\cap B'_n=\emptyset $.  By \refeq{Kr} we then have 
\begin{equation}
K_{\nu_n}(r)\leq \rho^n(\sum_{I\in \CI_n}K_{n,I}+ \sum_{I\in \CI_n(r)}\rho^{-n}|I|K_{n,I}):=\rho^n(L_n+L_n(r)),\qquad r\in (\rho^n,2\rho^n).
\label{Krbound}
\end{equation}
Thus inserting (\ref{Krbound}) into (\ref{CMn}) we get
\begin{equation}
\CM_n\geq\int_{\rho^n}^{2\rho^n}{1\over(L_n(r)+L_n)}{\rho^{-n}dr}.
\label{cmnbound}
\end{equation}

The term $L_n(r)$ in the integrand (\ref{cmnbound}) is independent of $\CM_m$, $m>n$.
However $L_n$ is not and we will decouple it now. From
 (\ref{Krbound}) and (\ref{Kbound}) we get
 \begin{equation}
L_n\leq \sum_{\NJ}\delta_{\nu_n}(\NJ)
\label{Ln1}
\end{equation}
where the sum runs over a set of $\NJ=(J_1,J_2)$ with $J_i\in\cup_{i=0,\pm 1}\CD_{np+5+i} $  and  $J_i\subset B_n$.
In
particular 
\begin{equation}
|J_i\setminus B'_n|\geq  2^{-np-7}=\rho^n2^{-7}.
\label{Jbound}
\end{equation}
The sum in (\ref{Ln1}) has an $n$-independent number of terms (with
multiplicities).

Next estimate $\delta_{\nu_n}(\NJ)$ in terms of  a $V^n_{B_n\setminus
B'_n}$ measurable term and perturbation: \qq
\delta_{\nu_n}(\NJ)&=&{\nu_n(J_1\setminus B'_n)+\nu_n(J_1\cap B'_n)\over
\nu_n(J_2\setminus B'_n)+\nu_n(J_2\cap B'_n)}+(1\leftrightarrow 2)\non\\
&\leq&{\nu_n(J_1\setminus B'_n)+\nu_n(J_1\cap B'_n)\over
\nu_n(J_2\setminus B'_n)}+(1\leftrightarrow 2)\non\\
&=&\delta_{\nu_n}(J_1\setminus B'_n, J_2\setminus B'_n)+ {\nu_n(J_1\cap
B'_n)\over \nu_n(J_2\setminus B'_n)}+{\nu_n(J_2\cap B'_n)\over
\nu_n(J_1\setminus B'_n)}. \non
\qqq
Then decompose the perturbation further downscale:
\begin{equation}
\nu_n(J_i\cap
B'_n)=\sum_{m=n+1}^\infty\nu_n(J_i\cap(B'_{m-1}\setminus B'_m))
\non
\end{equation}
and (recalling \refeq{calJ}) define 
\qq
L_{n,n}&=&\sum_{(J_1,J_2)\in {\mathcal J}(I),\; I\in\CI_{n}}\delta_{\nu_n}(J_1\setminus B'_n, J_2\setminus B'_n)\label{Lnn}\\
L_{n,m}&=&\sum_{(J_1,J_2)\in {\mathcal J}(I),\; I\in\CI_{n}}{\nu_n(J_1\cap(B'_{m-1}\setminus B'_m))\over
\nu_n(J_2\setminus B'_n)}+(1\leftrightarrow 2) \label{Lnm} \quad\mbox{for}\;\; m\geq n+1.\qqq
Then
\begin{equation}
L_n\leq\sum_{m=n}^\infty L_{n,m}.
\label{Lsum}
\non
\end{equation}
Defining
\begin{equation}
m_n=\int_{\rho^n}^{2\rho^n}{1\over (1+L_n(r)+L_{n,n})} \, \rho^{-n}dr
\label{mn}
\end{equation}
and using the inequality 
$$
L_n(r)+L_n\leq (1+L_n(r)+L_{n,n})(1+\sum_{m=n+1}^\infty L_{n,m})
$$
we get from  (\ref{cmnbound}) 
\begin{equation}
\CM_n\geq m_n \sigma_n
\label{cmnbound1}
\end{equation}
with
\begin{equation}
 \sigma_n:=(1+\sum_{m=n+1}^\infty L_{n,m})^{-1}
\label{cmnbound1}
\end{equation}
Combining this with (\ref{decoup}) we  arrive to the desired bound of $M_n$ in terms of
random variables localized in the white noise:
\begin{equation}
M_n\geq m_ns_n\sigma_n:=M_n'.
\label{Mnfinal}
\end{equation}

\prop{propo2} {\bf (i)}\quad The random variables $m_n$ are
${V^n_{B_n\setminus B'_n}}$ measurable, $0\leq m_n\leq1$, and they
form an independent set. Moreover, 
$$\prob (m_n\leq x)\leq Cx, \quad \mbox{for}\;\; x>0,$$
  where  $C$   is
independent on   $\rho$ and $n$.

\smallskip

\no {\bf (ii)}  There exists $a>0$, $q>1$  and $C<\infty$ (independent of $n,m$ and     
  $\rho$) such that for all $m>n\geq 1$ the random variable $L_{n,m}$  satisfies the estimate
\qq
\prob (L_{n,m}> \lambda)\leq C\lambda^{-q}\rho^{(m-n-1/2)(1+a)}.
\label{ellbound}
\qqq
Moreover, $L_{n,m}$ is
${V^n_{B_n\setminus B'_m}}$ measurable. Especially, $L_{n,m}$ and $L_{n',m'}$ are independent if $n>m'$ or $n'>m$.

 \eprop
The proof of this Proposition is postponed to Subsection
\ref{subse:proposition}.

\subsection{Law of large numbers and proof of Theorem \ref{mainestimate}}
\label{subse:lln}

Here we prove our main probabilistic estimate assuming 
Propositions  \ref{propo1} and \ref{propo2}. 
By (\ref{Mnfinal}) we need to consider
\beqla{eq:Pdef}
P_N:=\prob(\sum_{n=1}^N M'_n<N\delta)=\expec \KHI(\sum_1^Nm_ns_n\sigma_n\leq\delta N)=:\expec\KHI_{D_N},
\eeq
where we denoted $D_N:= {\{ \omega:\; \sum_1^Nm_ns_n\sigma_n\leq\delta N\}}$.
For the sake of notational clarity we  used above (and will often use later on) the shorthand 
$\KHI(A)$ for the indicator function $\KHI_A$.
In order to obtain the desired bound for $P_N$ we insert suitable auxiliary characteristic functions in the expectation. Define
\begin{equation}
\KHI_n:=\prod_{m=n+1}^\infty\KHI(L_{n,m}\leq 2^{n-m} \delta^{-1/4})
\prod_{m=0}^{n-1}\KHI(t_{n,m}\leq 2^{m-n}\log (\hf\delta^{-1/4})):=\prod_{m\neq n}\KHI_{n,m}.
\label{khin}
\end{equation}
On the support of $\KHI_n$ we have 
$$\sum_{m=n+1}^\infty L_{n,m}\leq \delta^{-1/4}
$$
and thus (for $\delta<1$ say)
$$
\sigma_n\geq \hf\delta^{1/4}.
$$
Similarily $\sum_{m=0}^{n-1}t_{n,m}\leq \log \hf\delta^{-1/4}$ and so
$$
s_n\geq 2\delta^{1/4}.
$$
Insert next
$$1=\prod_{n=1}^N(\KHI_n+(1-\KHI_n)):=\prod_{n=1}^N(\KHI_n+
\KHI^c_n)$$ 
 in the expectation in  (\ref{eq:Pdef}) and expand to get
$$
P_N=\sum_{A\subset\{1,\dots,N\}}\expec \KHI_{D_N}\KHI_A
\KHI_{A^c}^c
$$
where $\KHI_A=\prod_{n\in A}\KHI_n$ and $\KHI^c_{A^c}=\prod_{n\in A^c}\KHI^c_n$. 
On the support of $\KHI_{D_N}\KHI_A
\KHI_{A^c}^c$ one has
$$
N\delta\geq\sum_nm_ns_n\sigma_n\geq \delta^\hf\sum_{n\in A}m_n
$$
 so
\qq\label{Pchi}
P_N&\leq&\sum_{|A|> \alpha N}\expec \KHI( \sum_{n\in A}m_n\leq \delta^\hf N
)+\sum_{|A|\leq \alpha N}\expec\KHI^c_{A^c},
\qqq
where we choose $\alpha: =\min (1,a)/8$ with $a$ taken from Proposition \ref{propo2} (ii).
Observe that $\alpha$ is independent of $\rho, \delta$ and $N.$

Let us consider the two sums on the RHS of (\ref{Pchi}) in turn. For the first
one we use independence: let $m_A:=\sum_{n\in A}m_n$ then
\begin{equation}
P(m_A<\delta^\hf N)\leq e^{\delta^\hf tN}\expec e^{-tm_A}=e^{\delta^\hf tN}\prod_{n\in A}\expec e^{-tm_n}.
\label{ind}
\end{equation}
By Proposition \ref{propo2} (i)
\begin{equation}
\expec e^{-tm_n}\leq Cx+e^{-tx}\leq 2e^{-tx(t)}
\label{}
\end{equation}
where the auxiliary variable $x=x(t)$ is chosen so
 that $Cx(t)=e^{-tx(t)}$. Here $x(t)\to 0$ and $tx(t)\to\infty$ as $t\to\infty$. Thus assuming $\delta$ small enough and
taking $t=t(\delta)$ such that  $x(t)=2\delta^\hf/\alpha$, in the case $|A|\geq \alpha N$ the right side of
(\ref{ind}) is bounded
by $2^Ne^{-\delta^\hf t(\delta )N}$ where $\delta^\hf t(\delta)\to\infty$ as $\delta\to 0$.
Hence
\begin{equation}
\sum_{|A|>\alpha N}\expec \KHI( \sum_{n\in A}m_n\leq \delta^\hf N
)\leq 2^Ne^{-g(\delta)N}.
\label{1st}
\end{equation}
where $g(\delta)\to \infty$ as $\delta\to 0$.

 For the second sum in (\ref{Pchi})
we  need to bound
$$
\expec\KHI^c_{B}:=\expec\prod_{n\in B}(1-\KHI_n)
$$
for $|B|\geq (1-\alpha)N$. For that purpose, we shall make use of  the elementary identity
\begin{equation}\label{simple}
1-\prod_{j=1}^\infty (1-a_j)=\sum_{j=1}^\infty a_j\prod_{r=1}^{j-1}(1-a_r),
\end{equation} 
valid for any sequence $(a_j)_{j\geq 1}$ with $a_j\in  [0,1]$ for all $j\geq 1.$
Recall eq. (\ref{khin}) and denote $\KHI^c_{n,m}:=1-\KHI_{n,m}$ . We also set 
$\KHI^c_{n,m}:=0$ for $m<0.$ For any fixed $n$ 
arrange the variables $\KHI^c_{n,m}$ with $m\in\integer$ into a sequence
in some order
and apply the identity \refeq{simple} to write
\begin{eqnarray}
1-\KHI_n=1-\prod_{{m\in\integer},\; {m\neq n}}(1-\KHI^c_{n,m})
=\sum_{{\ell\in\integer},\; {\ell\neq 0}}\KHI^c_{n,n+l}\tilde\KHI_{n,\ell},
\end{eqnarray}
with certain variables $\tilde\KHI_{n,\ell}$ satisfying $0\leq \tilde\KHI_{n,\ell}\leq 1 $. Let us denote
\qq
\KHI_n^+
:=\sum_{\ell >0}\KHI^c_{n,n+\ell}\tilde\KHI_{n,\ell}\quad \mbox{and}\quad 
\KHI_n^-
:=\sum_{\ell <0}\KHI^c_{n,n+\ell}\tilde\KHI_{n,\ell}.
\qqq
Then $\KHI_n^\pm\leq 1$ (since $\KHI_n^++\KHI_n^-=1-\KHI_n$) and
\qq
\KHI_n^\pm\leq 
\sum_{\pm \ell>0}\KHI^c_{n,n+\ell}.
\label{expansion}
\qqq
We may then estimate
\qq
\prod_{n\in B}(1-\KHI_n)&=&\prod_{n\in B}(\KHI_n^++\KHI_n^-)=\sum_{(s_n=\pm )_{ n\in B}}\prod_{n\in B}\KHI_n^{s_n}\non\\
&\leq&\sum_{s:N_+>(1-2\alpha)N}\ \prod_{n:s_n=+}\ \KHI_n^{+}+
\sum_{s:N_+\leq(1-2\alpha)N}\ \prod_{n:s_n=-}\ \KHI_n^{-}
\label{product}
\qqq
where $N_+$ is the number of $n$ in the set $B$ such that $s_n=+$. We estimate
the expectations of the two products on the RHS in turn.

For the first product, let $D\subset \{1,\dots,N\}$ with $p:=|D|\geq ( 1-2\alpha)N$. List the
elements of $D$ as $n_1<n_2<\dots<n_p$. Then, as $0 \leq \KHI^+_{n_j} \leq 1$,
\begin{equation}
\expec \KHI^+_{n_1}\cdots \KHI^+_{n_{p}}\leq\sum_{\ell_1>0}\expec \KHI^c_{n_1,n_1+\ell_1} \KHI^+_{n_2}\cdots\KHI^+_{n_{p}} \leq
\sum_{\ell_1>0}\expec  \KHI^c_{n_1,n_1+\ell_1}\KHI^+_{n_{i_{2}} }\cdots \KHI^+_{n_{p}},
\label{1}
\end{equation}
where $n_{i_{2}}$ is the smallest $n_j$ larger than $n_1+\ell_1$.
Iterating we get
\begin{equation} \label{lisa3}
\expec \KHI^+_{n_1}\cdots \KHI^+_{n_{p}}
\; \leq \; \sum_{r=1}^p \; \sum_{(\ell_1,\dots,
\ell_r)}\expec \prod_{j=1}^r  \KHI^c_{n_{i_j},n_{i_j}+\ell_j}
\end{equation}
where $n_{i_{j+1}}$ is the smallest $n_j$ larger than $n_{i_{j}}+\ell_j$ and  $n_{i_{1}}=n_1$. As the intervals $[n_j,n_j+\ell_j]$ cover the set $D$, the $r$-tuples $(\ell_1,\dots,
\ell_r)$ in the above sum satisfy
\begin{equation}\label{lsum}
\sum_{J=1}^r \ell_j\geq p-r
\end{equation}

 Next, by Proposition \ref{propo2} (ii) the factors in the product in (\ref{lisa3})
are independent and thus
$$
\expec \prod_{j=1}^r  \KHI^c_{n_{i_j},n_{i_j}+\ell_j}=\prod_{j=1}^r \expec  \KHI^c_{n_{i_j},n_{i_j}+\ell_j}
$$
From (\ref{ellbound}) and (\ref{khin}) we deduce
$$
\expec  \KHI^c_{n_j, n_j+\ell_j}\leq C(\rho)\delta^{q/4}(2^q\rho^{1+a})^{\ell_j}
$$
whereby
\begin{equation}
\expec \KHI^+_{n_1}\cdots \KHI^+_{n_{p}}
\leq \sum_{r=1}^p \; \sum_{(\ell_1,\dots,
\ell_r)}(C(\rho)\delta^{q/4})^r(2^q\rho^{1+a})^{\sum \ell_j}  \label{4}
\end{equation}
Using (\ref{lsum})
we see that RHS is bounded by
$$
\rho^{(1+a/2)p}\sum_{r=1}^p \,(C(\rho)\delta^{q/4})^r \,\sum_{(\ell_1,\dots,
\ell_r)}(2^q\rho^{a/2})^{\sum \ell_j}
$$
For an upper bound drop the constraints on $\ell_i$ to bound (\ref{4}) by
$$
\rho^{(1+a/2)p}\sum_{r=1}^p(C(\rho)\delta^{q/4})^r
\bigl(\sum_{\ell=1}^\infty (2^{q}\rho^{ a/2})^\ell\bigr)^r
$$
Choosing first $\rho$ small enough and then  $\delta\leq\delta(\rho)$ this is bounded by 
$$
C(\rho)\delta^{1/4}\rho^{(1+a/2)p}\leq C(\rho) \delta^{1/4}\rho^{(1+a/2)(1-2\alpha)N}\leq C(\rho)
\delta^{1/4}\rho^{(1+2b)N}
$$
for a constant $b>0$ by our choice of $\alpha$. The expectation of the first sum in eq. (\ref{product})
is then bounded by
\qq
C(\rho) 2^N\delta^{1/4}\rho^{(1+2b)N}.
 \label{1sum}
 \qqq

Consider finally the second sum in eq. (\ref{product}). We proceed as for the
first sum this time considering a set $D\subset \{1,\dots,N\}$ with elements
$n_1>n_2>\dots>n_p$
with $p\geq \alpha N$.  Now we write $ \KHI^-_{n_1}\leq\sum_{\ell_1>0}\KHI^c_{n_1,n_1-\ell_1}$
and end up with the analogue of eq. (\ref{lisa3}):
\begin{equation}
\expec \KHI^-_{n_1}\cdots \KHI^-_{n_{p}}
\leq \sum_{r=1}^p\; \sum_{(\ell_1,\dots,
\ell_r)} \; \prod_{j=1}^r\expec  \KHI^c_{n_{i_j},n_{i_j}-\ell_j}
 \label{3}
\end{equation}
where $n_{i_{j+1}}$ is the largest $n_j$ smaller than $n_{i_{j}}-\ell_j$
and $n_{i_{1}}=n_1$, and this time Proposition \ref{propo1} was used 
for independence.
From the same  Proposition we also get 
$$
\expec  \KHI^c_{n,n-\ell}\leq c \, e^{-c2^{-2\ell}\rho^{-\ell+\hf}(\log \delta)^2}.
$$
For small enough $\rho$ we have $2^{-2\ell}\rho^{-\ell+\hf}\geq (\ell+\rho^{-{_1\over^8}})\rho^{-{_1\over^8}}$ for all $\ell\geq 1$. Hence
$$
\prod_{j=1}^r\expec  \KHI^c_{n_{i_j}, n_{i_j}-\ell_j}\leq c^r \, \exp\Bigl(-c(\log\delta )^2\rho^{-{_1\over^8}} \bigl(r\rho^{-{_1\over^8}}+\sum_{j=1}^r\ell_j \bigr)\Bigr)
$$
As $\rho<1$, by (\ref{lsum}) we also have 
$$
r\rho^{-{_1\over^8}}+\sum_{j=1}^r\ell_j \;  \geq \,(p+\sum_{j=1}^r\ell_j )/2.
$$
Thus
$$
\prod_{j=1}^r\expec  \KHI^c_{n_{i_j}, n_{i_j}-\ell_j}\leq
\exp\bigl(c\rho^{-{1\over 8}}(\log\delta )^2p/2\bigr) c^r\exp\bigl(-c(\log\delta )^2\rho^{-{_1\over^8}} \sum_{j=1}^r\ell_j \bigr)
$$
Now recall that $p\geq \alpha N$, take $\delta$ small enough, and proceed as above by summing first  over the $\ell_j$:s, and then performing a geometric sum over $r$ in order to conclude that
the second sum in  (\ref{product}) has the upper  bound  
$$
2\exp\bigl(-c\rho^{-{1\over 8}}(\log\delta )^2N\alpha/2\bigr) 
$$
For small $\delta$ this is by far dominated by the bound  (\ref{1sum}),
and therefore 
\qq
\expec\KHI^c_{B}\leq 2^{N+1}\delta^{1/4}\rho^{(1+2b)N}. 
 \label{1111}
 \qqq
Going back to equation  (\ref{Pchi}), and recalling  (\ref{1st}) with (\ref{1sum}) and (\ref{1111}), we conclude that
for $\delta\leq \delta(\rho)$
\qq
P_N\leq 2^{2N+2}\delta^{1/4}\rho^{(1+2b)N}. 
 \label{1112}
 \qqq
 which gives the claim of Theorem  \ref{mainestimate}.
 \hfill\halmos

\subsection{Proofs of the Propositions}
\label{subse:proposition}

We will now prove the Propositions \ref{propo1} and \ref{propo2} of Subsection
\ref{subse:correlation} describing the statistics of $m_n$,
$L_{n,m}$ and $t_{n,k}$. We start by noting that the random measures $\nu_n(\cdot)$ and $\rho^{n-1} \nu_1(\rho^{1-n}\cdot)$ are equal in
law. Especially, the $m_n$ are i.i.d. and it suffices to  study $m_1$.
Similarly $\zeta_k|_{B_n}$ equals in law with $\zeta_1|_{B_{n-k+1}}$
and thus $t_{n,k}$ equals $t_{n-k+1,1}$ in law. The value
$k=0$ is slightly different, but it can be treated exactly in the same manner as the case $k\geq 1.$ Finally,
$L_{n,m}$ and
 $L_{1,m-n+1}$  are equal in law. 

 We need first the following Lemma.
 \lem{le:delt1} 
There exists $q,q_1>1$ and $C>0$ {\rm(}each independent of $\rho${\rm)} such that
for all  intervals $J,I\subset [-1/4 , 1/4]$  satisfying
$|J|\leq 2 |I|$,
and with mutual distance at most $100|I|$, one has
\begin{equation}
\prob \bigl(\,\delta_{\nu}(J,I) >\lambda \bigr)\leq C\lambda^{-q}\left({|J|\over |I|}\right)^{q_1}.
\label{eps1}
\end{equation}
\elem 
\begin{proof}
We use the comparision \refeq{eq:difference3} with the measure $\eta$
in order to estimate
\beqla{eq:g2}
\nu(J)/\nu (I)\leq G_2 \,\eta(J)/\eta (I), 
\eeq
where we recall that all the moments of the variable $G_2$ are finite. Next, in case $|I|\leq 1/100$ we may scale further by using the exact scaling
law \refeq{eq:scaling2},  and apply the translation invariance of $\eta$ to deduce that $\eta(J)/\eta (I)\sim \eta(J')/\eta (I')$, where now 
$I',J'\subset [0,1] $ with $1/100 \leq  |I'|$ and $|J'|\leq |J|/|I| \leq 100 |J'|.$ In the case $|I|\geq 1/100$ no scaling is  needed. 

In this situation, if $r<\infty$ it follows from Proposition \ref{lisa5} that $\eta(I')^{-1}\in L^r$ uniformly with respect to the $I'$. We can thus  fix exponents $1<q<\widetilde q<p< 2/\beta^2$ and get by \refeq{eq:g2},   H\"older's inequality and Theorem \ref{th:fact1}
\begin{eqnarray}\label{lisa7}
\|\nu(J)/\nu(I)\|_{q}&\leq& C\|\eta(J)/\eta(I)\|_{\widetilde q}= C\|\eta(J')/\eta(I')\|_{\widetilde q} \\
&\leq& C\|\eta(J')\|_p\leq C(|J|/|I|)^{\zeta(p)/p},\nonumber
\end{eqnarray}
 where $\zeta(p)>1$. The constant $C$ depends only on the exponents $q,\widetilde q$ and $p$. Thus
\qq
\prob (\,\delta_{\nu}(J,I)> \lambda)\leq C\lambda^{-q}(|J|/|I|)^{q\zeta(p)/p}
\label{ellbound0}
\qqq
The desired bound follows by choosing the exponent $q>1$ close enough to $p$ in order to ensure that $q_1:=q\zeta(p)/p>1.$ 
 \hfill\halmos
\end{proof}
\medskip

 Let us then discuss $m_1$.
Observe that the denominator of the integrand in (\ref{mn}) can
be dominated as follows:
 \begin{equation}
1+L_{1,1}+L_1(r)\leq 1+L_{1,1}+\sum_{m=0}^\infty 2^{-m}k_m(r)\label{L1}
\end{equation}
where for $r\in (\rho,2\rho )$ and $m\geq 0$ one sets
\begin{equation}
k_m(r):=\sum_{I \in\CD_{p+m}}K_{1,I}1_{C_I\cap S_r\neq\emptyset}.
\label{kmr}
\end{equation}
 For any  fixed $r\in (\rho, 2\rho )$ the sum \refeq{kmr} has at most four non-zero terms.

For $m\geq 0$  denote by ${\mathcal H_m}$ the set of all pairs $\NJ =(J_1,J_2)$
that contribute to $k_m(r)$ in \refeq{kmr}
for some $r\in (\rho, 2\rho)$.  To estimate $\delta_{\nu_1}(\NJ)$, we may scale by the factor $\rho^{-\hf}$ in order to 
consider instead the identically distributed quantity $\nu(J'_1)/\nu(J'_2)$,
where  now $J'_1,J'_2 \subset [-1/4,1/4]$. Thus Lemma  \ref{le:delt1} applies. As  we additionally have $|J_1|=|J_2|$,  there is $q>1$ and a constant $C>0$ such that
\begin{equation}
\prob (\delta_{\nu_1}(\NJ)>R)\leq CR^{-q} \quad \mbox{for all}\quad
\NJ\in\cup_{m\geq 0} {\mathcal H_m}.
\label{eps2}
\end{equation}

 Choose next $\alpha>0$ and  $\gamma\in (0,1)$ such that $4\alpha\sum_m2^{m(\gamma-1)}\leq 1$
together with $\gamma q>1$.  Fix $R>0.$ We observe that by these choices
\begin{eqnarray*}
\delta_{\nu_1}({\NJ})\leq \alpha 2^{\gamma m}R \;\; \mbox{for all}\;\; \NJ\in {\mathcal H_m},\; m\geq 0\quad \Longrightarrow \quad L_1(r)\leq R \;\;\mbox{for all}\;\; r\in (\rho, 2\rho).
\end{eqnarray*}
Since  we have 
the obvious estimate $\# ({\mathcal H_m})\leq c2^m$ for the number of the
pairs in ${\mathcal H_m}$, by combining the above implication  with the uniform estimate \refeq{eps2} one may estimate
\begin{eqnarray*}
&&\prob (L_1(r(\sigma))>R \;\; \mbox{for some}\;\; r\in (\rho, 2\rho))
\;\leq \;\sum_{m=0}^\infty\sum_{\NJ\in {\mathcal H_m}}\prob (\delta_{\nu_1}(\NJ)>\alpha 2^{\gamma m}R)\non\\
&&\phantom{moooore}\leq\; CR^{-q}\sum_mc2^{m}2^{-q\gamma m}\quad 
\leq\quad CR^{-q}.
\end{eqnarray*}

In a similar vain we may apply Lemma \ref{le:delt1} to immediately obtain  the corresponding tail  estimate for  $L_{1,1}.$ Indeed, 
by \refeq{Lnn} this
depends only on a  finite ($\rho$-independent) number of ratios $\delta_{\nu_1}(I_1,I_2)$, with $I_1,I_2\subset [-4\rho ,4\rho]$
and  $|I_1|, |I_2|\geq 2^{-7}\rho,$  see (\ref{Jbound}).
Putting things together, we obtain (for  $R>1$, say) the bound
 \begin{equation}
\prob (m_1<1/R)\leq C R^{-q}\leq C R^{-1},
\label{ber}
\end{equation}
with $C$ is independent of $\rho .$

\vs{3mm}

Consider next $L_{n,m}$ with $m>n$ and use $L_{n,m}\sim L_{1,m-n+1}$.
By (82) $L_{1,m-n+1}$ is bounded from above by a sum of terms  (with $\rho$-independent upper bound for their number) 
$$\nu_1(J)/\nu_1(I)$$ 
where $2^{-8}\rho \leq |I| \leq 2^{-4}\rho $ and $|J|\leq \rho^{m-n+\hf}$, and in addition $I,J\subset [-4\rho,4\rho ]$.  The constant $C$ above is independent of $m,n$ and $\rho$. Via scaling the desired bound (\ref{ellbound})  is now a direct consequence of Lemma
\ref{le:delt1}, as we observe that $|J|/|I|\leq C \rho^{m-n-\hf}$.

\vs{3mm}

Finally we turn to  $t_{n,1}$ given in  (\ref{tn}). By scaling
we may take the
sup and the inf over $x\in B_n\cap\real$ of $\exp({\beta\widetilde\psi})$ where $\widetilde\psi:=\psi(\cdot,\rho^{3/2},
\rho^{{1/2}})$ and
we may 
replace there $\widetilde\psi$ by
 $\widehat\psi:=\widetilde\psi(\cdot)-\widetilde\psi(0)$.   The covariance of 
 $\widehat\psi$ is clearly $c\rho^{-3/2}$-Lipschitz
and length of the interval $B_n\cap\real$ is $8\rho^n.$
Lemma \ref{le:Talagrand} yields that
$$
\prob (|\widetilde\psi |> \lambda c \rho^{n/2-3/4})\leq
C(1+\lambda)e^{-\hf\lambda^2},
$$
which  finishes the proof of the remaining Proposition \ref{propo1}.\hfill\halmos

\subsection{Integrability of $K_\nu$}
\label{subse:L1}

In next section we shall also make use of the the following observation:

\lem{KL1} Let $\beta<\sqrt{2}$. Then almost surely $ K_\nu  \in L^1([0,1]\times [0,2] )$. 
\elem
\begin{proof}
Recall that  $S= \R \times [0,2]$ is tiled by the Whitney
squares
 $C_I$. By definition,
on such a square    $ K_\nu$ is a finite sum of  ratios 
 $\nu (J_1)/\nu (J_2)$ with $|J_1|=|J_2|\leq 2^{-4}$ and  of
 controlled mutual distance as in Lemma \ref{le:delt1}. Thus, for  $|J_i|$
 small enough $J_i$ lie on a common interval of length $\hf$ and we have a uniform bound for
 $\expec \nu(J_1)/\nu(J_2) \leq \|\nu(J_1)/\nu(J_2)\|_{q}$, $q < 2/\beta^2$ 
 from Lemma \ref{le:delt1} (or more directly from (\ref{lisa7})) and for the finitely many
 ones not fitting to such interval we use again (\ref{lisa7}).
Hence  there is also a uniform bound for  $\expec K_\nu (I)$ and
one obtains

\medskip

$
\hbox{}\qquad\qquad\displaystyle\expec \int_{[0,1]\times [0,2]} K_\nu  \leq \sum_{I\subset {D}([0,1])}|C_I|\;\expec {K}_\nu(I)\leq
C\sum_I|C_I|<\infty .
$ \hfill \halmos
\end{proof}
\section{Conclusion of the proof}
\label{se:conclusion}

In this final section we give a precise formulation to  our main result as a  Theorem and
prove it using the work done in the previous sections. In order to
make the setup clear, let us recall that our random circle homeomorphism
 was defined in Section \ref{se:circlehomeo} via formulae
 (\ref{randomhomeo1}) and (\ref{randomhomeo2}).
  Its extension  to the unit disc is constructed by the
method described in Section \ref{subse:extension}, and formula (\ref{homeo51}) in particular. 

The welding method described in Section 2  requires  estimates for  the Lehto integral of the distortion function in $\D$.  Theorem \ref{alaReed}  reduces these bounds to the boundary  function, and here the crucial
estimates are  provided by 
 our Theorem \ref{mainestimate} in Section \ref{se:probability}.

\thm{th:4.1} Let $\phi:\torus \to \torus$ be the  random circle homeomorphism from Definition \ref{def:randomhomeo}, and let $\Psi:\D \to \D$ be its extension as in (\ref{BA}) - (\ref{homeo51}).
Let $\mu = \mu_\Psi := \partial_{\bar z} \Psi / \partial_z \Psi$ be the complex dilatation of the
extension on $\D$, and set $\mu=0$
outside $\D$. 

Then almost surely there exists a (random) homeomorphic
$W^{1,1}_{loc}$-solution $f:\complex\to\complex$ to the Beltrami
equation
\beqla{eq:aux}
\dzb f =\mu\dz f \qquad \mbox{a.e.  in } \C,
\eeq
that satisfies the normalization $f(z)=z+ o(1)$ as $z\to\infty$.
Moreover, there exist $\alpha
>0$ such that the restriction $f:\torus\to\complex$ is a.s.
$\alpha$-H\"older continuous.
\ethm
\begin{proof} 
 We sketch the proof along the lines of \cite[Thm 20.9.4]{AsIwMa}, to which presentation we refer for  further details and background.
 
 \medskip
 

 For 
any integer $n\geq 1$
choose $N_n=[\rho^{-(1+\hf b)n}]\in\nanu$  where $b$ is as in Theorem \ref{mainestimate}. Denote
$$\zeta_{n,k}:=
\exp(2\pi ik/N_n)\;\;\mbox{for}\;\; k={1,\ldots ,N_n} .
$$
 Write also
$G_{n}:=\{ \zeta_{n,1},\ldots ,\zeta_{n,N_n}\}.$ Thus the distance on $\torus$
to the set $G_n$ is bounded by $2\pi /N_n\sim \rho^{(1+\hf b)n}$, up to
a constant.

For a given $n\geq 1$ and $k\in {1,\ldots N_n}$ let us denote by $A_{n,k}$ the event
$$
A_{n,k} = \{ \omega:  L_{{K}_\nu}
(k/N_n,\rho^n,2\rho)< n\delta \} ,
$$
and set $A_n=\bigcup_{k=1}^{N_n}A_{n,k}.$ Note that here we consider Lehto integrals in the half plane. 
Theorem \ref{mainestimate}(i) combined with stationarity yields that
\beq
\sum_{n=1}^\infty \prob (A_n)\leq \sum_{n=1}^\infty\sum_{k=1}^{N_n}
\prob (A_{n,k})\leq\sum_{n=1}^\infty N_n c(\delta )\rho^{(1+b)n}
\leq c(\delta)\sum_{n=1}^\infty \rho^{{b\over 2}n}<\infty .\nonumber
\eeq
Borell-Cantelli lemma yields that almost every $\omega$ belongs
to the complement of the event $\bigcup_{n>n_0(\omega )}A_n.$

Also, we obtain by Lemma
\ref{le:comparision} that
$$
{K}_\tau \leq E^2{K}_\nu,
$$
where almost surely $E<\infty.$
From Theorem 2.6  and \refeq{Koot}  we see that  $K(z,F)$,  the distortion of the extension of $h$, is bounded  by  a constant times $ K_\tau (z)$.  Hence
Lemma \ref{KL1} implies that almost surely
$$
 \int_{[0,1]\times [0,2]}K(z,F) \leq  C_0 \int_{[0,1]\times [0,2]} {K}_\tau \leq C_0 E^2 \int_{[0,1]\times [0,2]} {K}_\nu <\infty .
$$

We may thus  forget the probabilistic
setup by fixing an event $\omega_0\in\Omega$ so that we are in the following situaton: We are given the complex dilatation $\mu$ on $\D$, so that the distortion $ K = (1+|\mu|)/(1-|\mu|)$  satisfies pointwise 
$$
K( e^{2 \pi i z}) \leq C_0 {K}_\tau(z)\leq C_0 E(\omega_0)^2 {K}_\nu(z), \quad z \in \IH.
$$
Further, from the definition in (\ref{homeo51}) we have $K \equiv 1$ for $|z| \leq e^{-4\pi }$. We also have $K_\nu \in
L^1 \cap L^\infty_{loc}$ on the square $[0,1]\times [0,2]$, and     for each
$n\geq n_0$ and $k\in {1,\ldots ,N_n}$ it holds that
\beqla{eq:4.300}
  L_{{K}_\tau}(k/N_n,\rho^n,2\rho) \geq 
 (E(\omega_0))^{-2} L_{{K}_\nu}(k/N_n,\rho^n,2\rho)
\geq n\delta  (E(\omega_0))^{-2}=:\nonumber n\delta' . 
\eeq
\vskip10pt

 We next proceed as in
the standard proof of Lehto's theorem by  approximating $\mu$
by e.g. the sequence $\mu_\ell:=\frac{\ell}{\ell+1}\, \mu,$ $\ell \in \nanu$. Letting
$f_\ell$ denote the corresponding normalized solution of the
Beltrami equation with coefficient $\mu_\ell$, i.e. with the asymptotics  $f_\ell(z)=z+ o(1)$ as
$z\to\infty$, then every $f_\ell$ is a quasiconformal homeomorphism of $\C$.

To show that  \refeq{eq:aux} has a homeomorphic $W^{1,1}$-solution, we need to control the approximations $f_\ell$. For this  we first apply \cite[Lemma 20.2.3]{AsIwMa}, which tells that the inverse maps $g_\ell = f_\ell^{-1}$ have the following modulus of continuity,
$$|g_\ell(z) - g_\ell(w) | \leq 16\pi^2\,  \frac{|z|^2 + |w|^2 + \int_\D \frac{1+|\mu_\ell(\zeta)|}{1-|\mu_\ell(\zeta)|}\, d\zeta}{\log\bigl( e + \frac{1}{|z-w|} \bigr)}, \quad z,w \in \C.
$$
Here the integrals are uniformly bounded as 
$$ \frac{1+|\mu_\ell(\zeta)|}{1-|\mu_\ell(\zeta)|}  \leq  K( \zeta) \leq C_0  {K}_\tau(z), \quad \quad  \zeta = e^{2\pi i z},
$$
and $K_\tau \in
L^1[0,1]\times [0,2]$. Thus the inverse maps $g_\ell = f_\ell^{-1}$ form an equicontinuous family. 

In order to check
the equicontinuity of the family $(f_\ell)_{\ell \geq 1}$ itself we
first consider a point $z\in \D$. Writing $2a=1-|z|$, observe that
$K$ is bounded in $B(z,a)$ and as $K_\ell:= K(\cdot, f_\ell)
\leq K$ we have for any $\ell \geq 1$ and $u\in (0,a/2)$ \beq
L_{K_\ell}(z,u,1)&\geq& L_K(z,u,a)\geq (\| K\|_{L^\infty
(B(z,a))})^{-1}\log
(a/u)\nonumber\\
&&\to \infty\quad \mbox{as}\;\; u\to 0.\nonumber \eeq Moreover, by
Koebe's theorem or \cite[Cor. 2.10.2]{AsIwMa} we obtain
\beqla{eq:4.20}
f(2\D)\subset 5\D .
\eeq
Thus $\diam
(f_\ell (B(z,1)))\leq 5$, which may
be combined with Lemma \ref{le:4.1} to obtain
$$
\diam (f_\ell (B(z,u)))\to 0\quad \mbox{as $u\to 0$,\; \;uniformly in}\;\;\ell .
$$
This proves the equicontinity at interior points $z\in \D$. Equicontinuity at exterior points follows e.g.  form Koebe's theorem.

In order to next consider the uniform behaviour on $\torus$, note that it suffices to prove local equicontinuity on points of $[0,1]$ for the  family
$$F_\ell(z) = f_\ell(e^{2 \pi i  z}), \quad \ell \in \nanu.
$$
 We first
estimate the diameter of the image  $F_\ell\bigl(B(k/N_n, \rho^{n})\bigr),$
assuming that $n\geq n_0$. Applying  the fact $\diam F_\ell\bigl(B(k/N_n, 2\rho)\bigr) \leq  \diam (f_\ell
(B(\zeta_{n,k},1)))\leq 5$ and using this together with Lemma \ref{le:4.1} 
we obtain
 \beqla{eq:4.40}
 \diam (F_\ell
(B(k/N_n,\rho^{n}))) &\leq& \diam (F_\ell
(B(k/N_n, 2\rho)))16\exp(-2\pi^2n\delta')\\
&\leq& \nonumber 80 e^{-nc'}.
\eeq
From these estimates we get the required equicontinuity. 
Namely, working now on the circle $\torus$, since  the set $G_n$ is evenly spread on $\torus$,
the balls $B(\zeta_{n,k},\rho^{n+1})$ cover a $\rho^{n+2}$-neighbourhood
of $\torus$ in such a way that any two points that are in this
neigbourhood, with distance  not exceeding $\rho^{n+2}$, lie in the
same ball. Since this holds for every $n\geq n_0$ we infer from
\refeq{eq:4.40} that there is $\varepsilon_0>0$ and $\alpha >0$ so
that, uniformly in $\ell$ 
\begin{equation}\label{eq:4.50} 
|f_\ell (z)-f_\ell(w)|\leq
C|z-w|^\alpha\quad\mbox{if }\;\;|z|=1,\;\; \mbox{and}\;\; \left\{
\begin{array}{l}1-\varepsilon_0
\leq |w|\leq 1+\varepsilon_0\\
|z-w|\leq\varepsilon_0.
\end{array} \right. 
\end{equation}
 One may actually take $\alpha =c'/\log (1/\rho ) .$
This clearly yields equicontinuity at the points of $\torus$, and
hence on $\Chat$. We may now pass to
a limit and one obtains $W^{1,1}$-homeomorphic solution
$f(z)=\lim_{\ell\to\infty} f_\ell(z)$ to the Beltrami equation as in
\cite[p. 585]{AsIwMa}.

At the same time the estimate \refeq{eq:4.50} shows that
$f:\torus\to\complex$ is H\"older continuous. Since $f$ is analytic outside the disk, with $f(z) = z + o(1)$ at  infinity, in fact it follows that $f$ is  H\"older continuous on $\C \setminus \D$.  \halmos
\end{proof}



Collecting the  results established we now arrive at  the main theorem of this paper. 

\thm{main result} Let $\phi = \phi_\omega$  be the  random  circle homeomorphism, with derivative the exponentiated GFF, as defined  in  (\ref{randomhomeo1}) and (\ref{randomhomeo2}). 
 
 Then for $\beta^2< 2$ and almost surely in $\omega$, the mapping $\, \phi \, $ admits a  conformal welding. That is,  there  are a random Jordan curve  
 \beqla{eq:finally}
 \gamma = \gamma_{\omega,\beta}
 \eeq
  and conformal mappings $f_\pm$ onto the complementary domains of $\gamma$, such that $\phi = f_+^{-1} \circ f_-$ on 
$\torus$.

Moreover, almost surely in $\omega$,  the Jordan curve $\,\gamma\,$  in (\ref{eq:finally}) is unique, up to composing with a M\"obius transformation $\Gamma = \Gamma_\omega$ of the Riemann sphere. 
\ethm
\begin{proof}  We argue as in Section \ref{se:welding}. Using the complex dilatation of the extension $\Psi$ from Theorem \ref{th:4.1}, we find a homeomorphic solution $f$ to the auxiliary equation (\ref{eq:aux}). This is conformal outside the disk, so we set $f_- = f |_{ \C \setminus \D}$. Inside the disk $K(z,f)$ is locally bounded, so the uniqueness of the Beltrami equation gives $f(z) = f_+ \circ \Psi (z)$, $z \in \D$, where $f_+$ is a conformal homeomorphism on $\D$. Since the boundary $\partial f_+(\D) = \partial f_-(\C \setminus \D) =  f(\torus) = \gamma$ is a Jordan curve, $f_{\pm}$ extend to $\torus$ where we have
$$    \phi= (f_+)^{-1} \circ f_- .
$$
Finally, according to the proof of Theorem \ref{th:4.1} $f_-$ is H\"older continuous in $\C \setminus \D$, thus the uniqueness of the welding curve follows from the Jones-Smirnov Theorem \ref{homeo41}.
 \halmos
\end{proof}

\medskip

\section{Appendix: Negative Moments}
\label{se:Appendix 1}

Here we prove the finiteness of all negative moments
for the measure $\eta$, that was defined in the proof
of Theorem \ref{th:fact1}.

\begin{proposition} \label{lisa5}
Suppose $\beta^2 < 2$. Then 
$$\expec \bigl( \eta(I)^{-q} \bigr) \leq C < \infty, \quad \quad 0 < q < \infty,
$$
for a constant $C = C(q, |I|)$ depending only on the exponent $q$ and the length $|I|$.
\end{proposition}
\begin{proof} Fix $t>0$. Define for $\varepsilon >0$ the set $U_{\varepsilon,t}$ by setting 
$
U_{\varepsilon,t}:= U\bigcap \{ \varepsilon<y \leq t\} .
$
As in \refeq{eq:limitmeasure3} one deduces the existence of the limit measure
\beqla{eq:limitmeasure4} 
\eta_t (dx):=   \lim_{\varepsilon\to 0^+} \exp \bigl(\beta U_{\varepsilon,t}
 (x)-(\beta^2/2){\rm Cov}\, (U_{\varepsilon,t})\bigr)dx. \quad
                      \eeq
Denote $M:=\eta_{1/2}([0,1])$, $M_1:=\eta_{1/8}([0,1/4])$ and
$M_2:=\eta_{1/8}([3/4,1]).$ By scaling  and translation invariance  the random variables
$M_1,M_2$ and $M$ are identically distributed. Moreover, by comparing the exponents as
in the proof of Lemma \ref{le:difference}, we see that
\beqla{eq:MMM}
M\geq B(M_1+M_2),
\eeq
where $B:=\exp \bigl(\inf_{x\in[0,1]} \beta U_{1/8,1/2}(x)-(\beta^2/2){\rm Cov}\, (U_{1/8,1/2})\bigr)$ has all moments finite.
By construction, the random variables
$M_1,M_2$ and $B$ are independent.

Similarily, by comparing $\eta$ and $\eta_{1/2}$ we
see that it is enough to prove
\beqla{eq:enough}
\expec M^{-q}<\infty\qquad \mbox{for}\quad q>0.
\eeq
We first prove this for small values of $q$. For that end, 
consider for $s>0$ the Laplace transform
\beqla{eq:square}
\Psi_M(s)&:=&\expec \exp (-sM)\leq \expec (-sB(M_1+M_2))\\
&\leq& \expec \Psi_{M_1} (sB)\Psi_{M_2} (sB)=\expec \bigl(\Psi_{M}(sB)\bigr)^2.\nonumber
\eeq
Since especially $\expec B^{-1}<\infty $, we may  estimate
$\prob (B<1/s)\leq c/s.$ By substituting $s^2$ in place of $s$
in \refeq{eq:square} and applying this inequality we obtain
\beqla{eq:functionla}
\Psi_M(s^2)\leq c/s+ \Psi^2_M(s),
\eeq
  where one may assume that $c\geq 2.$

Denote $f(s):=(c/s^{1/2} +\Psi_M(s)).$ Then \refeq{eq:functionla} yields
\beqla{eq:functionla2}
f(s^2)=c/s+\Psi_M(s^2)\leq f^2(s).
\eeq
Since $\Psi_M(s)\to 0$ as $s\to\infty$ (while $\prob(M=0)=0$),
we may choose $s_0>0$ with $\Psi_M(s_0)\leq 1/2,$ whence
\refeq{eq:functionla2} iterates to $f(s_0^{2^k})\leq 2^{-2^k}$
for $k\geq 1.$ Together with monotonicity of $f$ this yields
$\delta >0$ such that  $f(s)\leq cs^{-\delta}$ for $s>0,$
especially $\Psi_M (s)\leq cs^{-\delta}$.

We obtain that
$$
\expec M^{-\delta/2}= c\int_0^\infty \expec e^{-sM}s^{\delta/2 -1}\, ds <\infty.
$$
In order to cover all values of $q$ in \refeq{eq:enough}
we employ a simple bootstrapping argument.
Assume that $\expec M^{-q}<\infty$ for some $q>0.$
By applying the inequality between the arithmetic and
geometric mean, the independence of $B, M_1$ and $M_2$, and the fact that $B$ has all negative moments finite, we may estimate
\beqla{eq:last}
\expec M^{-2q}&\leq &\expec (B(M_1+M_2))^{-2q}
\leq c\expec(M_1M_2)^{-q}= c(\expec (M)^{-q})^2 <\infty .
\eeq
By induction, this finishes the proof. \halmos
\end{proof}

\end{document}